\documentclass[10pt]{article}
\usepackage[utf8]{inputenc}
\usepackage[T1]{fontenc}
\usepackage[utf8]{inputenc}
\usepackage[english]{babel}
\usepackage[margin=3cm]{geometry}
\usepackage{amsmath, amssymb}
\usepackage{amsfonts}
\usepackage{dsfont}
\usepackage{float}
\usepackage{graphicx}
\usepackage{wrapfig}
\usepackage{mathtools}
\usepackage{bbm}
\usepackage{amsthm}
\usepackage{ifthen}
\usepackage{graphicx}
\usepackage{hyperref}
\usepackage[ruled,vlined]{algorithm2e}
\usepackage{xcolor}
\usepackage{mathtools}
\usepackage{empheq}
\usepackage{mathrsfs}
\usepackage{multicol}

\newtheorem{theorem}{Theorem}[section]

\newtheorem{definition}[theorem]{Definition}
\newtheorem{remark}{Remark}[section]
\newtheorem{lemma}[theorem]{Lemma}

\newtheorem{assumptions}{Assumptions}[section]

\newcommand{\ca}[1]{\mathcal{#1}}
\newcommand{\bb}[1]{\mathbb{#1}}

\newcommand{\set}[1]{\left\{#1\right\}}
\newcommand{\parent}[1]{\left(#1\right)}

\newcommand{\Rd}{\bb{R}^d}

\newcommand{\R}{\bb{R}}

\newcommand{\norm}[1]{\left\lVert#1\right\rVert}

\newcommand{\prom}[1]{\langle #1 \rangle}

\chardef\bslash=`\\ 

\numberwithin{equation}{section}

\newcommand{\N}{\mathbb{N}}

\def\bm{\left( \begin{array}{cc}}
\def\endm{\end{array}\right)}

 \newcommand{\Dv}{\operatorname{div}}
 
 \providecommand{\norm}[1]{\lVert#1 \rVert}

\newcommand{\be}{\begin{equation}}
\newcommand{\ee}{\end{equation}}
\newcommand{\ba}{\left(\begin{array}{c}}
\newcommand{\ea}{\end{array}\right)}
\newcommand{\bea}{\begin{eqnarray}}
\newcommand{\eea}{\end{eqnarray}}
\newcommand{\bee}{\begin{eqnarray*}}
\newcommand{\eee}{\end{eqnarray*}}
\newcommand{\ben}{\begin{enumerate}}
\newcommand{\een}{\end{enumerate}}

\title{The Calder\'on's problem via DeepONets}
\author{Javier Castro$^{a}$\and Claudio Mu\~noz$^{b,c}$\and Nicol\'as Valenzuela$^{b}$}
\date{\today}
\usepackage{tikz-cd}

\begin{document}

\maketitle

\begin{abstract}
We consider the Dirichlet-to-Neumann operator and the {\color{black} direct and inverse} Calder\'on{\color{black} 's} mappings appearing in the Inverse Problem of recovering a smooth bounded and positive isotropic conductivity of a material filling a smooth bounded domain in space. Using deep learning techniques, we prove that these mappings are rigorously approximated by DeepONets, infinite-dimensional counterparts of standard artificial neural networks
\end{abstract}

\begin{center}
{\it Dedicado a Carlos Kenig, con cari\~no y admiraci\'on}
\end{center}

{\footnotesize
\noindent
\thanks{$^{a}$: } Bielefeld University, Germany. Email address: {\tt jcastrom@math.uni-bielefeld.de}. \\
Funded by Chilean grant Fondecyt 1191412. \\
\thanks{$^{b}$: } Departamento de Ingeniener\'ia Matem\'atica DIM, Universidad de Chile. Email addresses: {\tt cmunoz@dim.uchile.cl, nvalenzuela@dim.uchile.cl}. 
Funded by Chilean grants Fondecyt 1191412, 1231250 and Basal CMM FB210005.\\
\thanks{$^{c}$: } Centro de Modelamiento Matem\'atico CMM, Universidad de Chile.
}

\tableofcontents 

\section{Introduction}

\subsection{Setting: the Calder\'on's problem} In this paper we deal with the Calder\'on's problem. This is the question posed by Calderón in \cite{C,Cal06}, on whether is it possible to determine the electrical conductivity in a medium by only perfoming measurements of the voltage on the boundary.  The background mathematical model is based on the following Dirichlet boundary value problem
\begin{equation}\label{eq:direct}
\begin{cases}
\text{div}(a(x)\nabla u)=0
&
\text{in }\Omega,
\\
u= f
&
\text{on }\partial\Omega,
\end{cases}
\end{equation}
being $\Omega\subset\R^d$, $d\geq 2$ a smooth bounded domain, $u:\overline{\Omega}\to\R$, $f:\partial\Omega\to\R$, and $a:\overline\Omega\to\R$, in the isotropic case, while $a:\overline\Omega\to\mathcal M_{d\times d}(\R)$, with $a$ symmetric in the anisotropic case. The coefficient $a$ plays the role of the {\it electrical conductivity}. For the purposes of this paper, we will only consider the isotropic case. It will be assumed that $a(x)$ is bounded away from zero, in the sense that $a(x)  \geq a_0>0$ a.e. in $\Omega$. 

\medskip

Let $H^{{\color{black} 1/2}}(\partial\Omega)$ denote the classical Hilbert space placed on the boundary, and $H^{-{\color{black}1/2}}(\partial\Omega)$ its dual. It is well-known that $H^{{\color{black} 1/2}}(\partial\Omega)$ is the image under the trace operator of the Sobolev space $H^1(\Omega)$. It is also known by the standard elliptic theory that, given $f\in H^{{\color{black} 1/2}}(\partial\Omega)$, and $a\in L^{\infty}(\Omega)$ there exists a unique weak solution $u\in H^1(\Omega)$ to \eqref{eq:direct}, that is
\[
\int_\Omega a \nabla u\cdot\nabla v\,dx=0,
\] 
for any $v\in H^{-1}(\Omega)$, and $u\vert_{\partial\Omega}\equiv f$, in the trace sense. As a consequence, given $a\in L^{\infty}(\Omega)$, one can define the {\it Dirichlet-to-Neumann operator}
\begin{equation}\label{eq:DN}
\Lambda_{a}:H^{{\color{black} 1/2}}(\partial\Omega)\to H^{{\color{black} -1/2}}(\partial\Omega),
\qquad
\Lambda_a f= a\left.\frac{\partial u}{\partial\nu}\right\vert_{\partial\Omega},
\end{equation}
where $u$ is the unique solution to \eqref{eq:direct} and $\nu=\nu(x)$ is the outward normal at the point $x\in\partial\Omega$. Of course, \eqref{eq:DN} makes no sense pointwise, but later we will give a precise rigorous definition making possible its use in applications. In this framework, the above question posed by Calderón can be rephrased in terms of the one-to-one property of the {\bf Calder\'on's mapping} $a\mapsto\Lambda_a$ in \eqref{eq:DN}, namely whether 
\begin{equation}\label{eq:question}
\Lambda_{a_1}=\Lambda_{a_2} \qquad \implies \qquad a_1=a_2,
\end{equation}
with $a_1,a_2\in L^{\infty}(\Omega)$.

\subsection{Brief review of results on the Calder\'on's problem} 

The answer to the original Calder\'on's problem is nowadays quite complete or almost satisfactory, after continuous improvements along the last 30 years. First, Kohn and Vogelius in \cite{KV} proved that for smooth conductivities, one can recover their values and those of their derivatives on the boundary. Later on, Alessandrini improved this result by proving  in \cite{A} that the Lipschitz conductivity at Lipschitz boundaries can be recovered by their Dirichlet-to-Neumann maps. 

\medskip

As for the recovering in the interior of the domain, the breakthrough has been produced by Sylvester and Uhlmann in the well celebrated works \cite{SU86,SU87}, where they first gave a positive answer to \eqref{eq:question}, proving that $C^2$-conductivities can be recovered inside the domain by boundary measurements, provided $d\geq3$. In this paper, they also introduced a quite general strategy based on the construction of suitable test functions, which are usually called CGO-solutions, introduced by Calderón himself in \cite{C}. Some years later, Brown \cite{Br96} improved the result by Sylvester and Uhlmann, again in dimension $d\geq3$, gaining half derivative on the regularity of the conductivities, requiring to be in Sobolev spaces $W^{\frac32,p}(\Omega)$, for suitable $p$. The important applied problem of reconstructing the conductivity from the provided data was fully addressed by Nachman \cite{Nac88}. In 2006, Astala and P\"aivarinta gave in \cite{AP} a complete positive answer to \eqref{eq:question} in the plane, proving that bounded conductivities can be recovered. The peculiar features of the plane, and the connection with the analogous problem of unique continuation led Uhlmann to conjecture that a threshold for the validity of \eqref{eq:question} in dimension $d\geq3$ should be represented by Lipschitz conductivities ($W^{1,\infty}$). Recently, a fundamental result in this direction has been proved by Haberman and Tataru in \cite{HT}, where small Lipschitz conductivities have been handled. See also the improvements for non{\color{black}-}Lipschitz conductivities in \cite{Hab15} for dimensions $3\leq d\leq 6$. Finally, Caro and Rogers removed the smallness condition in \cite{CR}.

\medskip

A related and more general question than \eqref{eq:question} is related to the possibility to recover the conductivity in the interior by boundary measurements performed {\bf only on a portion of the boundary}. In this case, the situation is quite more complicated, but similar results to the above are now available. A first important paper is due to Bukhgeim and Uhlmann, where a measurement on a bit more of the half of the boundary is required. Here, the authors make a strong use of Carleman's inequalities in order to construct suitable CGO-solutions. Later, in the well celebrated paper \cite{KSU}, {\bf Kenig, Sj\"ostrand and Uhlmann} introduced the tool of limiting Carleman weights and gave a strong improvement to the results in \cite{BU}, showing that measurements on even very small sets on the boundary permit to recover the conductivity. Later on, several improvements, especially related to the regularity of the conductivity have been produced, see e.g. \cite{UY} and the references therein.

\medskip

Consequently, one can say then that the Calder\'on's problem is well-understood in the isotropic setting, at least for sufficiently smooth conductivities. The very weak stability of the problem is also well-understood since the work by Alessandrini \cite{A}. The anisotropic case is much less understood, and the reader can consult the references \cite{Gu_survey} for further details. For the corresponding Calder\'on's problem in parabolic models, see \cite{Isakov}. 

\subsection{Main goal}

In this paper, we propose a different but not less interesting approach to face the Calder\'on's problem. This point of view goes as follows: in recent years, deep learning techniques have been intensively used to determine solutions to many applications, either numerically or theoretically speaking. The purpose of this paper is to propose and prove a rigorous mathematical justification for the use of deep neural networks in the simplest class of {\color{black} direct and inverse} Calder\'on's  mappings. Our approach is based on the use of the recently introduced DeepONets.  

\subsection*{Organization of this paper} This paper is organized as follows. In Section \ref{Sect:0} we review recent techniques in mathematical DNNs applied to PDEs. Section \ref{Sect:1} is devoted the main tools used in this paper, DeepONets. Next, Section \ref{Sect:1b} states the main results proved in this work. In Section \ref{Sect:2} we briefly review standard results about the Calder\'on's problem that will be needed along the paper. Section \ref{Sect:3} is devoted to additional and adequate characterizations of the Dirichlet-to-Neumann operator in the context of DNNs. In Section \ref{Sect:4} we fully introduce the required DeepOnets and describe its main properties. In Section \ref{Sect:5} we prove the main results of this paper. Finally, Section \ref{Sect:6} discusses several issues related to numerical implementation of the main results, among other important topics. 

\subsection*{Acknowledgments} We are deeply indebted of professors Miguel A. Alejo and Luca Fanelli, for their stimulating help with the inverse problems bibliography, data collected, and comments on a first version of this paper. {\color{black} Several comments by Gunther Uhlmann considerably helped to improve the point of view from inverse problems. Last but not least, we thank the referees for their comments and suggestions that finally led to a much better version of this manuscript.}

\section{Deep Learning techniques in PDEs}\label{Sect:0}

\subsection{Brief overview}

Deep Learning (DL) computational methods are being considered as a powerful numerical approach to solve linear and nonlinear PDE problems. Among these PDE problems, the approximation of PDE solutions is highlighted as one the most relevant ones \cite{ATY+19}, because its high relevance and impact in different applied and theoretical fields. Precisely, this approximation of an infinite dimensional object (the solution itself) is constructed by a new finite dimensional object that ``learns'' how to adapt itself to better fit the actual solution in view of the constraints prescribed by reality. The most important candidates to deep objects are precisely deep neural networks. 

\medskip

An artificial neural network (NN) is, roughly speaking, a generalized polynomial with a certain number of hidden layers and adjustable parameters that, depending on these degree freedoms, approximates a desired function (see \eqref{theta} for a precise, mathematical definition). The approximation context can be then placed as an improved version of the famous Stone-Weierstrass theorem, and precisely from this setting the finite dimension theory of neural networks is well-known from decades ago, see e.g. \cite{Hor91,LLPS93,Ros58,Yar16}. These results state that NNs are \emph{universal} approximators of continuous and measurable functions, or in mathematical terms, their set is dense in the continuous functions vector space. In theory, NNs with multiple hidden layers, known as \emph{deep} neural networks (i.e. more unknown parameters) are theoretically better suited to approximate a given function, but this is not the standard case in numerical applications. In reality, Stochastic Gradient Descent (SGD) type algorithms are used to find an appropriate setting for the unknown parameters of the NNs, precisely by using sets of data where one has good information about their possible outcome. Whether or not these algorithms give the exact right choice of approximative NN is still under debate, but important advances have been made during the past decade in this direction, see e.g. \cite{EHJ17}.

\medskip

Coming back to a PDE problem, when the dimension of the approximation spaces is high $d\gg1$, for the current numerical methods (finite differences, finite elements, among others) the complexity of the problem grows exponentially. This problem is currently known as the \emph{curse of dimensionality}. Many PDE problems in practice are posed in higher dimensions, such as those in finance and many problems describing the dynamics and behavior of multiple species or agents. It turns out that recent developments in DL may be highly useful to attack this problem. Indeed, the foundational articles \cite{HJE18,EHJ17,HPW19,RK18}, among others, propose to solve this issue by means of Deep Learning (DL) techniques. 




\medskip

Rigorous DL techniques in PDEs are classified in many categories, each one interesting by itself. We mention here some of the most related to this work: Monte Carlo, Multi-level Picard, Deep Galerkin, PINNs methods, among others. The domain grows fast and continuously, and current techniques may become outdated soon. Foundational results can be found in the works \cite{LLP,LLF}. Monte Carlo algorithms are an important and widely used approach to the resolution of the high dimension problem. This can be done by means of the classical Feynman-Kac representation that allows us to write the solution of an elliptic-parabolic linear PDE as an expected value, and then approximate the high dimensional integrals with an average over simulations of random variables. See \cite{Val22,Val23} and references therein for further results in this direction.  It is important to notice that the techniques presented in this paper may be useful to consider models where the Feynman-Kac representation previously mentioned is not available, such as some dispersive models. On the other hand, Multilevel Picard method (MLP) is another approach and consists on interpreting the stochastic representation of the solution as a modified fixed point equation, to then replace the obtained terms by suitable DNN approximations. This technique has showed quite powerful to consider many models of high interest. See \cite{BHH+20, HJK+20,EHJ+21,HJKN20} for fundamental advances in this direction. As another option, the so-called Deep Galerkin method (DGM) is another DL approach used to solve quasilinear parabolic PDEs plus boundary and initial conditions.  See \cite{SS18} for the development of the DGM and \cite{MNdH20} for an application.  In \cite{EHJ17}, E, Han and Jentzen proposed an algorithm for solving parabolic PDE by reformulating the problem as a stochastic control problem. This connection also come from the Feynman-Kac representation, proving once more that stochastic representations are a key tool in the area. More recent developments in this area can be found in Han-Jentzen-E \cite{HJE18} and Beck-E-Jentzen \cite{BEJ19}. Another important method has been devised by H\^ure, Pham and Warin in \cite{HPW19}, where they construct and prove consistency of forward-backward stochastic PDEs schemes using deep learning techniques, and applied their results to a huge family of parabolic PDEs. This approach is very robust and it was then extended by one of us \cite{Cas21} to the case of processes with jump discontinuities, representing nonlocal models. The Physics Informed Neural Networks (PINN) technique has become one of the most relevant methods to approximate solutions to physical models using neural networks. Originally proposed by Kardianakis and Raissi \cite{RK18}, the PINN technique uses the fact that some physical systems do have additional constraints that will help to fully recover the solution up to an error by quite fast DNN methods.  See \cite{RK18,RPK19,PLK19,MJK20} and references therein for a detailed introduction to the subject. This method has been highly influential, and in \cite{PINNsZurich} the authors have been able to provide rigorous justifications for the PINNs method. Indeed, they proposed an algorithm based on physics-informed neural networks (PINNs) to efficiently approximate solutions of nonlinear dispersive PDEs such as the KdV-Kawahara, Camassa-Holm and Benjamin-Ono equations. The stability of solutions of these dispersive PDEs is leveraged to prove rigorous bounds on the resulting error. 

\subsection{Deep Learning techniques in Inverse Problems}
 
It is well-known that the Calder\'on's problem is part of the huge family of inverse problems appearing in sciences and engineering. An inverse problem is, roughly speaking, any situation where one has recollected data (outputs) obtained from known inputs and wants to find an inner, unknown property or component of the main system. Since DNNs allow one to estimate functions from recollected data, they are naturally linked with inverse problems either from the numerical, practical and mathematical point of views. 

\medskip

There are many important results in the literature concerning DNN and Inverse problems. They range from applications, numerical, and theoretical results. We will review some of them, the interested reader can consult the references within these publications for a deeper insight. See also \cite{OJM22} for a complete introduction to the problem in terms of deep learning. Many results are very recent and some of them are still under review. Let us first mention the result in \cite{GCC22}, a transformer-based deep direct sampling method is proposed for a class of boundary value inverse problems. In ref. \cite{VAK+22}, the authors formulate a class of physics-driven deep latent variable models (PDDLVM) to learn parameter-to-solution (forward) and solution-to-parameter (inverse) maps of parametric partial differential equations (PDEs). 

\medskip

The work in \cite{FY20} is related to our results. The authors propose compact neural network architectures for the forward and inverse maps for both 2D and 3D problems. Numerical results demonstrate the efficiency of the proposed neural networks as well. In \cite{Li20}, the authors established a complete convergence analysis for the proposed NETT (network Tikhonov) approach to inverse problems. 

\medskip

From a more mathematical point of view, there are several important recent results for Inverse Problems in terms of NNs. In \cite{PKL22}, the authors established sharp
characterizations of injectivity of fully-connected and convolutional ReLU layers and networks. Additionally, the authors in \cite{dHLW21} developed a theoretical analysis for special neural network architectures, for approximating nonlinear functions whose inputs are linear operators. Such functions commonly arise in solution algorithms for inverse boundary value problems. See also \cite{AVLR21} for more results in this direction.

\medskip

There is a natural comparison between previous results and ours. Here we perform a more functional-analysis approach to the problem, seeking for the full approach or resolution to the Calder\'on's mapping instead of resolving for/from some particular inputs. This method will have advantages since we overcome the classical nearly-unstable character of Inverse Problems. On the other hand, our approach also has disadvantages: is less explicit in nature, highly theoretical sometimes; and we do not provide a specific algorithm to find the DeepONet. However, we have been able to provide a suitable form to compute it in some particular cases.

\section{First introduction to DeepONets}\label{Sect:1}

In all the previous results the use of well-defined finite dimensional deep neural networks is the main argument to solve the problem. However, many other problems require the approximation of functions between infinite dimensional spaces, such as a solution mapping, or any other complex higher dimensional object. For instance, the full solution map to a PDE. In the following subsection we introduce DeepONets, a generalization to DNNs.

\subsection{General continuous Banach case}

Some words first about the recent literature. Generalizing neural networks to an infinite dimensional framework has been recently of great interest. Earlier but fundamental results go back to Sandberg \cite{San91}, who defined a set of infinite dimensional mappings parameterized by finite dimensional parameters, providing a first universal approximation theorem. Probably one of the main results in this setting is the one given by Chen and Chen \cite{CC95}, who dealt with the approximation of functionals defined on a compact subset of $C(K)$, the real-valued continuous functions defined on $K$, with $K$ a compact of $\Rd$. They also considered more general mappings defined on $C(K)$ with values into $C(K)$. Among their results, a key result is Lemma $7$ in \cite{CC95}, one of the important properties stated in this lemma goes as follows:

\begin{lemma}[\cite{CC95}]\label{L3p1}
Let $(\varepsilon_k)_{k\in\N}$ be any collection of strictly positive numbers such that $\varepsilon_k\to 0$. Let $V$ be a compact set of functions in $C(K)$. Then, for every $k$ there exists an operator $T_k\colon C(K)\to C(K)$ such that every function $f\in V$ is $\varepsilon_k$-close to $T_k f$ in uniform norm sense. Moreover, the set
\begin{align*}
	V\cup\left(\bigcup_{k\in\N} T_k V\right)
\end{align*}
is compact in $C(K)$.  
\end{lemma}

{\color{black}L}emma \ref{L3p1} strongly relies on equicontinuity of compact sets in $C(K)$ and the construction of a partition of unity for every $k$. For every $k$, the transformed set $T_k V$ is constituted by simpler functions that can be easily described by finite dimensional neural networks, a property which allowed them to create the required DNN architecture. Chen and Chen also demonstrate that their architectures approximate any continuous mapping with respect to the uniform norm.

\medskip

{\color{black}
More recently Lu, Jin and Karniadakis \cite{LJK19}, strongly inspired by \cite{CC95}, introduced an architecture called {\bf DeepONets}. Said architecture is based upon the design of two sub-networks, namely the branch-net which deals with the input space, and the so called trunk-net for the output space. These sub-networks handle the discretization procedure for such functional spaces. Although the architecture presented here is built for Hilbert spaces, we decided to still call them DeepONets since we essentially follow the same philosophy as that of \cite{LJK19} and even more that of \cite{CC95}.

\medskip

In this section we shall introduce neural networks as they are commonly known. We then provide basic definitions concerning their infinite dimensional generalization. We use the same notation as in \cite{Cas22} in the infinite dimensional context. For further results, important to understand the ideas contained in this paper, the reader can consult the references \cite{LMK21, CC95}. 
}

\subsection{Finite dimensional deep neural networks}

The following definition introduces the notion of neural networks in finite-dimensional spaces. All previous results and approaches share that they use well-defined \emph{finite dimensional} DNNs as the main argument to solve the problem. However, many other problems need the approximation of functions between \emph{infinite dimensional} spaces, such a solution mapping, or any other complex object. Explicitly, the notion of neural networks in finite-dimensional spaces is as follows:  denote by $C(\mathbb R,\mathbb R)$ the set of real-valued continuous functions.

\begin{definition}
    We define
    \be\label{theta}
    \begin{aligned}
      \ca{N}_{\sigma,L,d,m}:= &~{}
     \Big\{ \theta=(A_1,A_2,\ldots,A_L) ~ : ~ \hbox{$A_j$ are linear affine functions}, \\
    &~{} \quad  \hbox{$A_j: \R^{d_j} \to \R^{d_{j+1}}$, $j=1,\ldots, L-1$, and $d_1=d$, $d_L=m$}  \Big\}, \\
    \ca{N}_{\sigma,d,m} :=&~{} \bigcup_{L\in \N} \ca{N}_{\sigma,L,d,m},
     \end{aligned}
     \ee
     as the set of neural network parameters with $L$ fixed or arbitrary number of hidden layers, respectively. 
\end{definition}
Above, as already mentioned, $L$ will represent the {\bf fixed number of hidden layers}. The dimension $d$ represents input size, the output dimension is $m$, and the {\bf activation function} $\sigma\colon\R\to\R$ must be ensured to satisfy the classical universality results for finite-dimensional DNN (i.e, $\sigma$ is not a polynomial). With a slight abuse of notation, we call $\theta$ as \emph{the} artificial neural network, different from its realization as a continuous function, which is introduced now:  

\begin{definition}
Given $\theta=(A_1,A_2,\ldots,A_L)\in\mathcal N_{\sigma,L,d,m}$, the corresponding realization of the neural network $\theta$ is the function $f^\theta: \R^d \to \R^m$ given by 
\[
f^{\theta}(x):=A_L \sigma A_{L-1} \sigma  \cdots \sigma A_1 x,
\]
and where $A_1, \ldots, A_L$ are suitable linear affine functions $A_j: \R^{d_j} \to \R^{d_{j+1}}$, $j=0,\ldots, L-1$, and $d_1=d$, $d_L=m$.
\end{definition}

\subsection{Infinite dimensional deep neural networks}

As it was indicated above, many PDE problems need a generalization of NNs in an infinite dimensional framework. In the cases studied in \cite{Cas22}, a general DL setting was needed, essentially to consider the approximation of nonlinear mappings $G: H\rightarrow W,$ where $H,W$ are separable Hilbert Spaces. With this result, the infinite dimensional Kolmogorov model was considered. Under this setting, the approximation of $G$ by a function $F^\theta: H\rightarrow W$ depending only on $\theta$ with $\theta\in \R^\kappa,~~\kappa\in \N$, is the main characteristic of the so call{\color{black}ed} \emph{DeepONets}. 

\medskip

Before introducing DeepONets, we need some additional ingredients. These are the {\bf projection and extension mappings}, the main idea behind these auxiliary functions is to lower the level of complexity into finite dimensional spaces. 



\medskip 

\noindent
{\bf Setting}. We would like to approximate operators mapping some Hilbert space $(H,\prom{\cdot,\cdot}_H,\norm{
\cdot}_H)$ to another $(W,\prom{\cdot,\cdot}_W,\norm{
\cdot}_W)$. Let $(e_n)_{n\in\bb{N}}$ and $(g_n)_{n\in\bb{N}}$ be orthonormal basis in $H$ and $W$ respectively.

\medskip

Consider the maps (recall that $d$ does not represent the dimension of the Calder\'on's elliptic PDE, and $\hbox{dim} W \geq m$, {\color{black} but instead the input size of the finite dimensional DNN included in the DeepONet})

\begin{minipage}{.45\linewidth}
	\begin{align*}
		P_{H,d}: H&\longrightarrow \Rd\\
		x &\longmapsto  \bigg(\prom{x,e_i}_H\bigg)_{i=1}^d,
	\end{align*}
\end{minipage}%
\begin{minipage}{.45\linewidth}
	\begin{align}\label{EE}
		E_{W,m}: \R^m&\longrightarrow W\nonumber \\
		a &\longmapsto \sum_{i=1}^m a_i g_i.
	\end{align}
\end{minipage}

\medskip	
	
Notice that $P_{H,d}$ is the standard projection operator which depends on the chosen orthonormal basis $(e_n)_{n\in\N}$. Additionally, $E_{W,m}$ extends finite dimensional vectors to the (possibly infinite dimensional space) $W$ by using the first $m$ members of the orthonormal basis $(g_n)_{n\in\N}$. 
	
	Now we are ready to fully define DeepONets.
	
	\begin{definition}\label{def:deeponets}
	Let $(H,d,\theta,m,W)$ be a fixed set of parameters; with $\theta\in\ca{N}_{\sigma,L,d,m}$ defined in \eqref{theta}, and $f^\theta:\R^d \to \R^m$ be its realization as finite dimensional deep neural network of parameter $\theta$. We define the DeepOnet $F^{H,d,\theta,m,W}:H\to W$ by
	\begin{align}\label{eq:DO-def}
		F^{H,d,\theta,m,W} = E_{W,m}\circ f^{\theta}\circ P_{H,d}.
	\end{align}
	See Fig. \ref{DON}.
\end{definition}
Notice that although $F^{H,d,\theta,m,W}$ is non-linear, it has finite range. Additionally, it is also continuous as a mapping from $H$ to $W$.

\begin{figure}[h!]
\centering
\begin{tikzcd}
H \arrow{r}{G} \arrow{d}[swap]{P_{H,d}} & W  \\   
\R^d \arrow[swap]{r}{f^{\theta}} & \R^m \arrow{u}[swap]{E_{W,m}}
\end{tikzcd}
\caption{Schematic description of how the DeepONet \eqref{eq:DO-def} is approximating a general mapping $G:H\to W$.}\label{DON}
\end{figure}

\medskip

	Unless is extremely necessary, we omit $H,W$ and write $F^{d,\theta,m}$ in Definition \ref{def:deeponets} to avoid a parameter overload. Set
	\begin{align*}
		\ca{N}^{H\to W}_{\sigma} &= \bigcup_{d,m\in\bb{N}} \set{d}\times\ca{N}_{\sigma,d,m}\times\set{m},\\
		\ca{N}^{H\to W}_{\sigma,L} &= \bigcup_{d,m\in\bb{N}}\set{d}\times\ca{N}_{\sigma,L,d,m}\times\set{m}.
	\end{align*}
	Observe that $\ca{N}^{H\to W}_{\sigma,L}\subseteq\ca{N}^{H\to W}_{\sigma}$ (the less parameters specified, the bigger the set). Let $\ca{N}=\ca{N}^{H\to W}_{\sigma}$ or $\ca{N}=\ca{N}^{H\to W}_{\sigma,L}$. It is straightforward to define
	\begin{align*}
		\ca{R}(\ca{N}) = \set{F^{H,d,\theta,m,W}\ \Big|\ (d,\theta,m)\in\ca{N}},
	\end{align*} 
as the set of continuous realizations of DeepONets mappings from $H$ to $W$. 

{\color{black}
\begin{remark}[On quantitative bounds for the projection operator $P_{H,d}$]
Motivated by the quest of finding specific and explicit bounds on the error obtained by DeepONets, we first recall the following uniform result proved by Daniilidis: 
\begin{lemma}[\cite{Cas22}]\label{lemma:aris}
Let $K$ be a compact set in $H$, and  $P_{H,d}u=\sum_{n\in\N} \langle u, e_n\rangle_H e_n$. Then for any $\varepsilon>0$ there exists $d\in \N$ such that for any $u\in K$, $\|P_{H,d} u - u\|_{H} <\varepsilon.$
\end{lemma}
Notice the uniformity of the bound for elements in the compact set $K$. This is essential in many aspects, and probably necessary if one wants to obtain specific bounds. In our case, 	$H=W_{\mu}=L^2(H^{1/2}(\partial \Omega) \times H^{1/2}(\partial \Omega),\mu \otimes \mu,\R).$ 
\end{remark}
}

	


\begin{remark}

In this paper we will work with Sobolev spaces which are separable Hilbert spaces. Consequently, from the previous result we will consider probability measures with finite second moment. 

\medskip

Recall the setting already introduced in \eqref{eq:DN}. In view of the Dirichlet-to-Neumann operator, and more generally, the Calder\'on's mapping, there are two (highly different) Hilbert settings in this paper: first, the space $L^2(\Omega)$, where the conductivity $a$ will be placed later on, and second, the classical space $H^{1/2}(\partial\Omega)$ and its dual space $H^{-1/2}(\partial\Omega)$, suitable Hilbert spaces for the Dirichlet-to-Neumann operator. More precisely, in the case of the space $H^{1/2}(\partial\Omega)$, we will consider $\mu$ be a finite measure with finite second moment satisfying 
    \[
    \int_{H^{1/2}(\partial\Omega)} \left(1+ \|f\|^2_{H^{1/2}(\partial\Omega)} \right)d\mu(f)<+\infty.
    \]
\end{remark}

\section{Main Results}\label{Sect:1b}

Recall the main ingredients in this paper. Let $\Omega$ be a bounded, smooth domain in $\mathbb R^d$, $d\geq 2$. In its simplest form, the standard and well-known Calder\'on's problem is the determination of the conductivity $a=a(x)\in L^\infty(\Omega) $ satisfying $a(x)\geq a_0>0$ in the elliptic problem
\be\label{Calderon}
\Dv_x [ a(x)\, \nabla u(x)] =0, \quad x\in \Omega, \qquad u \big|_{\partial\Omega} =f,
\ee
under the knowledge of the \emph{Dirichlet-to-Neumann map} (DN)
\be\label{DNmap}
\Lambda_a : H^{1/2}(\partial\Omega) \ni f  \longmapsto a ~\nabla u \cdot \nu \big|_{\partial\Omega} \in H^{-1/2}(\partial\Omega), \qquad \hbox{$\nu =$ unit outer normal to $\Omega$}.
\ee

\medskip

First of all, we introduce the background space for the conductivities $a$:
\begin{definition}[Admissible conductivities]\label{def:X-def}
   We define the set $X(\Omega)$ as follows,
    \begin{align}\label{XOmega}
        X(\Omega) = \set{a \in L^{\infty}(\Omega) ~ \colon ~ \exists ~a_0>0 \quad \text{s.t.} \quad a \geq a_0~\hbox{a.e.}}.
    \end{align}
\end{definition}
Notice that $X(\Omega)$ is continuously embedded into the Hilbert space $L^2(\Omega)$, thanks to the boundedness of $\Omega$. No other assumption on the conductivity $a(x)$ is needed.

\medskip

Having defined $X(\Omega)$, we introduce the Calder\'on's mapping:
\be\label{Lambda}
\Lambda : X(\Omega) \longrightarrow \mathcal L(H^{1/2}(\partial\Omega) , H^{-1/2}(\partial\Omega) ), \qquad \Lambda (a):= \Lambda_a,
\ee
where $\mathcal L(H^{1/2}(\partial\Omega) , H^{-1/2}(\partial\Omega) )$ is the vector space of continuous mappings from $H^{1/2}(\partial\Omega)$ towards $H^{-1/2}(\partial\Omega)$ (the fact that $\Lambda_a$ belongs to this space will be properly stated below). {\color{black}We often denote the norm of $H^{\pm1/2}(\partial\Omega)$ as $\norm{\cdot}_{\pm 1/2}$ and the norm of bounded linear operators as $\norm{\cdot}_{op}$.}

\subsection{Approximation of $\Lambda_a$, $a$ fixed}

Our first result shows that the   Dirichlet-to-Neumann operator $\Lambda_a$ described in \eqref{DNmap} is well-approximated by DeepONets:
\begin{theorem}\label{MT1}
       Let $a\in X(\Omega)$ and $\Lambda_a$ be as in \eqref{DNmap}. Let $\mu$ be any finite measure with finite second moment in $H^{1/2}(\partial\Omega)$. Then, for all $\varepsilon>0$ there exist a DeepONet $F^{\theta}_\varepsilon\colon H^{1/2}(\partial\Omega)\to H^{-1/2}(\partial\Omega)$ such that
	\begin{equation*}
		\int_{H^{1/2}(\partial \Omega)} \norm{\Lambda_a f - F^{\theta}_\varepsilon f}_{-1/2}^2 \mu(df) < \varepsilon,
	\end{equation*}
\end{theorem}
The previous result can be recast as follows: for every finite measure of finite second moment, the Dirichlet-to-Neumann operator for a fixed conductivity $a\in X(\Omega)$ can be approximated to any order of accuracy $\varepsilon$ in the $L^2(\mu)$ sense by a DeepONet with fixed parameters $\theta$. The precise description of this parameters is given in  Theorem \ref{MT1-general-case} below. Additionally, no further regularity for the boundary values are required other than $H^{1/2}$ and its dual space. 

\medskip

It is unclear to us whether or not the much desired uniform estimate
\[
\exists C_0>0 \quad \hbox{such that} \quad  \sup_{f\in H^{1/2}(\partial \Omega)}\norm{\Lambda_a f - F^{\theta}_\varepsilon f}_{-1/2} <C_0\varepsilon,
\]
does hold. This is due to the fact that we only have the injection $L^\infty \longrightarrow L^2$. However, since the $L^2$ convergence implies convergence in measure, one can state that for every $\lambda>0$ there exists a sequence of DeepONets $(F^{\theta_n}_{n})_{n\in\N}$ such that
\[
\lim_{n\to +\infty}\mu \left(\left\{ f \in  H^{1/2}(\partial \Omega) ~: ~ \norm{ \Lambda_a f - F^{\theta_n}_n f}_{-1/2} \ge \lambda \right\}\right) =0,
\]
and, up to a subsequence, one has almost convergence pointwise:
\[
\lim_{n\to+\infty} \norm{ \Lambda_a f - F^{\theta_n}_n f}_{-1/2}  =0, \quad \hbox{$\mu$-a.e. for $f$ in $H^{1/2}(\partial \Omega)$}.
\]
This means that with high probability, any $f\in H^{1/2}(\partial\Omega)$, $\Lambda_a f$ is close to a suitable approximation $F^{\theta_n}_{n}$.    


\medskip

Finally, the last aspect that we would like to mention is that the measure $\mu$ requires some particular properties, but it is highly general in perspective. Consequently, there is no need to ensure any particular additional configuration for the way of measuring the error term.

\subsection{Approximation of the full Calder\'on's mapping $a\mapsto \Lambda_a$}

Now we enunciate our second result, which proves the approximation of the full Calder\'on's mapping by infinite dimensional DeepOnets. Here, {\color{black}$W_{\mu}=L^2(H^{1/2}(\partial\Omega)\times H^{1/2}(\partial\Omega), \mu\otimes\mu; \R)$ is the vector space of measurable, real valued mappings $T$ defined on $H^{1/2}(\partial\Omega)\times H^{1/2}(\partial\Omega)$ such that 
\be\label{norma_del_demonio}
\|T\|_{W_\mu}^2 := \int_{H^{1/2}(\partial\Omega)} \int_{H^{1/2}(\partial\Omega)}| T(f,g) |^2 (\mu\otimes\mu)(df,dg)<+\infty. 
\ee
}
Notice that this space is suitable to measure the Dirichlet-to-Neumann operator $\Lambda_a: H^{1/2}(\partial\Omega) \longrightarrow H^{-1/2}(\partial\Omega)$. Later, in Lemma \ref{lemma:Lambda-linear-bounded}, we will sketch a proof that 
\begin{itemize}
\item for every $a\in X$, $\Lambda_a\in\ca{L}(H^{1/2}(\Omega), H^{-1/2}(\Omega))$, the topological vector space of linear continuous mappings, and
\item in Lemma \ref{Lem6p1} we will prove for any $\mu$ finite and of finite second moment, we will have 
\be\label{finitud}
\|\Lambda_a \|_{W_{\mu} } <+\infty.
\ee
\end{itemize}
This particular last result will provide the correct Hilbert-based setting for the Calder\'on's mapping.
\medskip

Recall that $d\geq 2$ is the dimension of the set $\Omega\subseteq \R^d$ under which the PDE associated to the Calder\'on's problem is placed. Using the norm in \eqref{norma_del_demonio}, we prove the following characterization of the Calder\'on's mapping via DeepONets:
\begin{theorem}\label{MT2}
Let $m>\frac{d}2$ and assume $a\in X(\Omega)\cap H^m(\Omega)$. Let $\mu$ and $\eta$ be any finite measures on $H^{1/2}(\partial \Omega)$ and $H^m(\Omega)$, respectively, with finite second moment. Let $\Lambda$ be the Calder\'on's mapping given by $\Lambda=(\Lambda_a)_{a\in X(\Omega)}$ in \eqref{Lambda} and satisfying \eqref{finitud}. Then, for all $\varepsilon>0$ the following is satisfied: there exists a  DeepOnet $\theta$ and its corresponding realization mapping 
\[
\mathcal F^{\theta}_\varepsilon  \colon X(\Omega) \cap H^m(\Omega) \to W_{\mu},
\]
such that
\be\label{eq:DO_Lambda_0}
\int_{X(\Omega)\cap H^{m}(\Omega)} \norm{\Lambda_a - \mathcal F^{\theta}_\varepsilon(a)}^2_{W_{\mu}} \eta(da) < \varepsilon.
\ee
\end{theorem}

\begin{remark}
The DeepONet $\mathcal F^{\theta}$ has a particular structure, deeply inspired in \cite{Cas22} and previous references. For more details, see \eqref{eq:DO-def}.
\end{remark}

Theorems \ref{MT1} and  \ref{MT2} are proved in Section \ref{Sect:5}.

\medskip

There are interesting consequences that can be obtained from \eqref{eq:DO_Lambda_0}. Unfortunately, they do not solve the classical Calder\'on's question \eqref{eq:question}, but give interesting insights on this problem. Following the same idea as in the Dirichlet-to-Neumann mapping,  
\[
\eta \left(\left\{ a \in  X(\Omega) \cap H^m(\Omega) ~: ~ \norm{\Lambda_a - \mathcal F^{\theta}_\varepsilon(a)}_{W_{\mu}} >\lambda \right\}\right) < \frac{\varepsilon}{\lambda^2},
\]
revealing that for large $\lambda$ and most conductivities $a$, the Calder\'on's mapping $\Lambda$ is well-approximated by the DeepONet $ \mathcal F^{\theta}_\varepsilon$, and for {\it most of the conductivities} $a\in X(\Omega)\cap H^m(\Omega)$  (in a set of size at most $\varepsilon\epsilon^{-2}$) one has the pointwise estimate
\be\label{lomejor}
\norm{\Lambda_a - \mathcal F^{\theta}_\varepsilon(a)}_{W_{\mu}} <\epsilon.
\ee
Then one is led to the problem of finding the inverse of the DeepONet $ \mathcal F^{\theta}_\varepsilon$, which is in principle a much less complex object than the original conductivity.

{\color{black}
	\subsection{Approximation of the inverse Calderón's mapping}
	
	Our last result is a rigorous approximation of the inverse Calderón's mapping in a particular setting of bounded conductivities. {\color{black} For this, we give two different approaches of the inverse Calderón's mapping, each one with its corresponding Theorem. The first approach is classical and assumes the separability of a complicated Hilbert space.
		
		\begin{theorem}\label{MT3}
			Let $\mu$ and $\nu$ be probability measures on $H^{1/2}(\partial \Omega)$ and $L^2(H^{1/2}(\partial \Omega),\mu;H^{-1/2}(\partial \Omega))$, respectively, with finite second moment. Suppose additionally that $\mu$ is such that $L^2(H^{1/2}(\partial \Omega),\mu;H^{-1/2}(\partial \Omega))$ is a separable space. For $M \in (0,\infty)$ let $\Sigma_M$ be the inverse Calderón's mapping, i.e. the mapping $\Lambda_a \mapsto a$, restricted to $Y_M(\Omega)$, a subset of conductivities of bounded size, defined as in \eqref{eq:inverse_operator}. Then, for all $\varepsilon > 0$ there exists a DeepOnet $\theta$ with realization $F^{\theta}_{\varepsilon}$ such that
			\begin{equation}\label{eq:MT3}
				\int_{\Lambda(Y_M(\Omega))} \norm{\Sigma_M(T) - F^{\theta}(T)}^2 \nu(dT) < \varepsilon,
			\end{equation}
			where $\Lambda(Y_M(\Omega))$ is the image of $Y_M(\Omega)$ under the direct Calderón's mapping $a \mapsto \Lambda_a$, viewed as a subset of $L^2(H^{1/2}(\partial \Omega), \mu;H^{-1/2}(\partial \Omega))$.
		\end{theorem}

The second and last result is another characterization of the inverse Calder\'on's mapping via DeepONets. This time we profit of the product space and the representation of distributions in $H^{-1/2}(\partial\Omega)$ via the Riesz's theorem.
		
		\begin{theorem}\label{MT4}
			Let $\mu$ and $\nu$ be probability measures on $H^{1/2}(\partial \Omega)$ and $W_{\mu}$, respectively,  with finite second moment. For $M \in (0,\infty)$ let $\Sigma_M$ be the inverse Calderón's mapping, i.e. the mapping $\Lambda_a \mapsto a$, resctricted to $Y_M(\Omega)$, a subset of conductivities of bounded size, defined as in \eqref{eq:inverse_operator}. Then, for all $\varepsilon>0$ there exists a DeepOnet $\theta$ with realization $F^{\theta}_{\varepsilon}$ such that
			\begin{equation}
				\int_{D} \norm{\Sigma_M \circ \chi (B) - F^{\theta}_\varepsilon (B)}_{L^2(\Omega)}^2 \nu(dB) < \varepsilon,
			\end{equation}
			where $D$ is a proper subset of $W_{\mu}$, where the bilinear continuous mappings $B$ have a linear representation which is a Dirichlet-to-Neumann mapping, and $\chi$ represents the operator that gives the linear representation (from $H^{1/2}(\partial \Omega)$ into $H^{-1/2}(\partial \Omega)$) of the bilinear continuous mapping $B$ (from $H^{1/2}(\partial \Omega) \times H^{1/2}(\partial \Omega)$ into $\R$). 
		\end{theorem}
	}
	
	These last results shows that in standard physical situations, the full inverse Calder\'on's problem admits a suitable approximation via DeepONets.
}

\subsection{Method of proof}

The rest of this paper is devoted to the proof of Theorems \ref{MT1}, \ref{MT2} and \ref{MT3}. Their proofs are highly related and require a key property: being able to establish each mapping (Calder\'on{\color{black}'s}, Dirichlet-to-Neumann) as a function defined in a possible high dimensional Hilbert space. This is unclear from the beginning, mainly because linear maps need not be included in square integrable spaces, for instance. The situation is even worse in the case of nonlinear maps. Here there is a key role played by the measures taken into account to measure errors. It will become clear from the proof that the hypothesis of existence of a finite measure with finite second moment is highly necessary, and overcoming or improving this condition is a necessity in future applications of the method. However, even in this restricted setting, Theorem \ref{MT1} and \ref{MT2} show that highly complex mappings such as those involved in the simplest inverse problems are well-approximated via DeepONets. Theorem \ref{MT3} will be a suitable improvement of previous results.

\medskip

Another key instance in the proof is Lemma \ref{Lem6p1}, which establishes the conditions under which one can embed the Calder\'on's mapping into a square-integrable Hilbert vector space. Then we will make use of a general result proved by one of us in \cite{Cas22}, stated here in Theorem \ref{vago}, to complete the proof of approximation. The obtained DeepONets can be further characterized in terms of dimensions, number of required hidden layers, and other aspects that will become clear in the proof itself.

%
%
%

\section{A brief review on the Dirichlet-to-Neumann operator}\label{Sect:2}

In this section we review a few results on Calderón{\color{black}'s} problem. We follow the notation of \cite{Sal08}. These results will allow us to express the Calder\'on's mapping as an approximation problem via deep learning tools.
First of all, we will review the classical continuity of the Dirichlet-to-Neumann map.

\subsection{Weak formulation of the PDE (\ref{Calderon})}
Let $\prom{\cdot,\cdot}$ denote the duality paring between a space and its dual. Let $f\in H^{1/2}(\partial\Omega)$, via integration by parts one has the classical weak formulation of equation \eqref{Calderon},
\begin{equation}\label{weak-fomulation1}
    \left\{ 
        \begin{aligned}
            \prom{a\nabla u,\nabla v} &= 0,\ \forall v\in H_0^1(\Omega),\\
            u|_{\partial\Omega} &= f,\ \text{in the trace sense}.
        \end{aligned}    
    \right.
\end{equation}
Now, since $f\in H^{1/2}(\partial\Omega)$, there exists $w\in H^1(\Omega)$ such that $w|_{\partial\Omega}=f$ and $\norm{w}_{H^1(\Omega)}\le C\norm{f}_{H^{1/2}(\partial\Omega)}$ for some $C>0$ independent of $f$ (this is the surjectivity of the trace operator $H^{1}(\Omega)\to H^{1/2}(\partial\Omega)$). Being $w$ fixed, consider now the following problem:
\begin{equation}\label{weak-fomulation2}
    \left\{ 
        \begin{aligned}
            \prom{a\nabla \Tilde{u},\nabla v} &= -\prom{a\nabla w, v},\ \forall v\in H_0^1(\Omega),\\
            \Tilde{u}|_{\partial\Omega} &= 0,\ \text{in the trace sense}.
        \end{aligned}    
    \right.
\end{equation}
Provided $a\in L^{\infty}(\Omega)$ and that exists $a_0>0$ such that $a\ge a_0$, {\color{black} one can check that the hypotheses of Lax-Milgram's theorem are met}. Thus, equation \eqref{weak-fomulation2} has a unique solution $\Tilde{u}\in H^1_0(\Omega)$. Then, $u:=\Tilde{u} + w\in H^{1}(\Omega)$ is the unique solution to \eqref{weak-fomulation1} and satisfies $\norm{u}_{H^1(\Omega)}\le C\norm{f}_{H^{1/2}(\partial\Omega)}$ for some $C>0$. This means that the problem \eqref{Calderon} is well-posed, and in view of this, it will be helpful to define the set of all $a$'s for which \eqref{Calderon} is well-posed. See Definition \ref{def:X-def} above.

\medskip

As it was stated in the introduction of this article, the Calderon's problem consists in finding the conductivity function $a$ from the knowledge of the Dirichlet-to-Neumann operator $\Lambda_a\colon H^{1/2}(\Omega)\to H^{-1/2}(\Omega)$ in \eqref{DNmap}. However, this mapping needs a rigorous definition in the setting established for \eqref{weak-fomulation1} and \eqref{weak-fomulation2}. This is done in Definition \ref{def:Lambda}. 

\begin{definition}\label{def:Lambda} 
    Let $a\in X(\Omega)$. Let $f,g\in H^{1/2}(\Omega)$, and let $u$ be the solution $u$ to \eqref{weak-fomulation1} with $f$ as boundary value. Let also $v\in H^1(\Omega)$ be any function such that $v|_{\partial\Omega}=g$ in the trace sense. The Dirichlet-to-Neumann operator $\Lambda_{a}\colon H^{1/2}(\Omega)\to H^{-1/2}(\Omega)$ is defined as
    \begin{align}\label{weak_Lambda}
        \Lambda_a f \in H^{-1/2}(\Omega), \qquad g\in H^{1/2}(\Omega), \qquad  \Lambda_a f (g) = \int a\nabla u\nabla v.
    \end{align} 
\end{definition}
This precise definition of the Dirichlet-to-Neumann operator will be key in forthcoming sections, see e.g. Lemma \ref{prop:L2mapping}.


\subsection{Boundedness of the Dirichlet-to-Neumann mapping}

Lemma \ref{lemma:Lambda-linear-bounded} below is a classical result in the Calderon's problem literature (see e.g. \cite[Lemma $3.4$]{Sal08} for an equivalent case). We sketch the proof for the sake of completeness, and some bounds will be useful later on.

\begin{lemma}\label{lemma:Lambda-linear-bounded}
    For any $a\in X(\Omega)$ as in \eqref{XOmega}, $\Lambda_a\in\ca{L}(H^{1/2}({\color{black}\partial}\Omega), H^{-1/2}({\color{black}\partial}\Omega))$, {\color{black}the space of} linear bounded operators from $H^{1/2}({\color{black}\partial}\Omega)$ into $H^{-1/2}({\color{black}\partial}\Omega)$.
\end{lemma}
\begin{proof}
    Let $a\in X$. $C>0$ denotes a constant that depends only on $(d,\Omega)$ and may change from one line to another.
    
    \medskip
    
    \noindent
    {\bf $\Lambda_a$ well-defined}. Let $f\in H^{1/2}({\color{black}\partial}\Omega)$. We shall prove that, according to Definition \ref{def:Lambda}, $\Lambda_a f$ is an element of $H^{-1/2}({\color{black}\partial}\Omega)$. Let $g\in H^{1/2}({\color{black}\partial}\Omega)$, we first claim that the real number $\Lambda_a f(g)$ does not depend on the choice of $v$ in Definition \ref{def:Lambda}. Indeed, let $v_1,v_2\in H^1(\Omega)$ be such that $v_1|_{\partial\Omega}=v_2|_{\partial\Omega}=g$, since $u$ is the unique solution to \eqref{weak-fomulation1} and $v_1-v_2\in H^1_0(\Omega)$, we know that
    \begin{align*}
        \int a\nabla u\nabla (v_1 - v_2) = 0\implies\int a\nabla u\nabla v_1 = \int a\nabla u\nabla v_2.
    \end{align*}
    Therefore $\Lambda_a f$ is well-defined as a function. Now, let $g_1,g_2\in H^{1/2}(\Omega)$, $v\in H^1(\Omega)$ such that $v|_{\partial\Omega}=g_1+g_2$ and $v_1,v_2\in H^1(\Omega)$ such that $v_1|_{\partial\Omega}=g_1$ and $v_2|_{\partial\Omega}=g_2$. Since $(v_1+v_2)|_{\partial\Omega} = v|_{\partial\Omega} = g_1 + g_2$, we get that $\Lambda_a f( g_1 + g_2) = \Lambda_a f(g_1) + \Lambda_a f(g_2)$. The linearity follows by noting that $ \Lambda_a f(0) = 0$. 
    
    \medskip
    
    \noindent
    Now we prove continuity. Let $g\in H^{1/2}(\partial\Omega)$, then
    \begin{align*}
       | \Lambda_a f(g) |&= \left| \int a\nabla u \nabla v \right|\le \norm{a}_{\infty} \norm{u}_{H^1(\Omega)} \norm{v}_{H^1(\Omega)}\le C\norm{a}_{\infty} \norm{f}_{1/2} \norm{g}_{1/2}.
    \end{align*}
    Therefore
    \begin{align}\label{Lambda-f-bound}
        \norm{\Lambda_a f}_{-1/2}\le C\norm{a}_{\infty}\norm{f}_{1/2},
    \end{align}
    which implies that $\Lambda_a f\in H^{-1/2}(\partial\Omega)$ and we have the right to write $\Lambda_a f = \prom{\Lambda_a f,\cdot}$.
    
    \medskip
    
    \noindent
    {\bf Linearity}. Let $f_1,f_2\in H^{1/2}(\partial\Omega)$, and let $u\in H^1(\Omega)$ be the associated solution to $f_1+f_2$ and $u_1,u_2\in H^1(\Omega)$ the associated solutions to $f_1,f_2$ respectively. By the linearity of the trace operator and the uniqueness of the solution of \eqref{weak-fomulation1}, we get that $u = u_1 + u_2$ and then for any $g\in H^{1/2}(\partial\Omega)$, 
    \begin{align*}
        \prom{\Lambda_a (f_1+f_2), g} = \int a\nabla u\nabla v =  \int a\nabla (u_1+u_2)\nabla v =  \prom{\Lambda_a f_1, g} + \prom{\Lambda_a f_2, g}.
    \end{align*}
    The linearity follows since $\Lambda_0 = 0$.
    
    \medskip
    
    \noindent
    {\bf Continuity}. Continuity follows from \eqref{Lambda-f-bound}:
    \begin{align}\label{Lambda-a-bound}
        \norm{\Lambda_a}_{\ca{L}(H^{1/2}(\Omega), H^{-1/2}(\Omega))} \le C\norm{a}_{\infty}.
    \end{align}
    This finishes the proof.
\end{proof}

\begin{remark}\label{remark:X-set}
    The main reason behind the definition of $X(\Omega)$ in \eqref{XOmega} is that we want to interpret $\Lambda$ as a mapping defined over a Hilbert space. Since $\Omega$ is bounded, and therefore of finite Lebesgue measure, one can see that $X(\Omega)$ is included in the space
    \begin{align}\label{eq:X-L2}
      X_2(\Omega):=  \set{a \in L^2(\Omega) ~ \colon ~ \exists ~ a_1,a_0>0 \quad \text{s.t.} \quad a_0 \leq a \leq a_1~\hbox{a.e.}}.
    \end{align}
 Moreover, one has $X(\Omega)=X_2(\Omega)$.  Thus from now on we work with $X(\Omega)$ as a subset of the Hilbert space $L^2(\Omega)$ as in \eqref{eq:X-L2}. 
 \end{remark}

\begin{remark}
Note that from Remark \ref{remark:X-set} and the proof of Lemma \ref{lemma:Lambda-linear-bounded}, one has that Lemma \ref{lemma:Lambda-linear-bounded} is also valid if the data on the boundary and $a$ in $X(\Omega)$ have more regularity; indeed, any Hilbert space $H^m(\Omega)$, $m\geq 0$, is suitable ambient space for our proofs, provided the boundary data is chosen in $H^{m-\frac12}(\partial\Omega)$.
\end{remark}

Being a bounded linear operator, $\Lambda_a$ has not enough properties to be easily approximated by DNN. In the next section, we will prove that $\Lambda_a$ can be understood as a square integrable operator in suitable Hilbert spaces.

{\color{black} \subsection{Invertibility of the Dirichlet-to-Neumann mapping}
	
	As we said in previous sections, there exist a variety of ``physically motivated'' conditions on the conductivity $a$ in order to have a well-definition and suitable properties of the Dirichlet-to-Neumann mapping $\Lambda_a$. For this end, let $M \in (0,+\infty)$ and define the space $Y_M(\Omega)$ as follows:
	\begin{itemize}
		\item If $d=2$, \quad $\displaystyle Y_M(\Omega) := \left\{ a \in L^{\infty}(\Omega): \; \frac{1}{M} \leq a \leq M \right\}$;
		\item If $d>2$, \quad $\displaystyle Y_M(\Omega) := \left\{ a \in W^{1,\infty}(\Omega): \; \frac 1M \leq a \leq M \right\}$.
	\end{itemize}
	Notice that for all $M \in (0,+\infty)$, the Calderón's mapping restricted to the space $Y_M(\Omega)$ is an injective operator, that is to say, $\Lambda$ satisfies \eqref{eq:question}. This is done using only the fact that $a \in L^{\infty}(\Omega)$ (in the case $d=2$) or $a \in W^{1,\infty}(\Omega)$ (in the case $d >2$), see e.g. \cite{HT,CR,AP}. On the other hand, the condition
	\[
	\frac 1M \leq a \leq M,
	\]
	is a more or less standard assumption on the conductivities, and it has a physical and in-real-life sustenance. In this context, due the injectivity in $Y_M(\Omega)$, if we restrict the Calderón's mapping to the image of $Y_M(\Omega)$, that restricted operator is bijective. In other words, the operator
	\[
	\widetilde{\Lambda}_M : Y_M(\Omega) \longrightarrow \Lambda(Y_M(\Omega)) \subset \mathcal L(H^{1/2}(\partial \Omega),H^{-1/2}(\partial \Omega)),
	\]
	is bijective. Denote by $\Sigma_M$ the inverse operator of $\widetilde{\Lambda}_M$, defined as
	\begin{equation}\label{eq:inverse_operator}
		\begin{matrix}
			\Sigma_M : & \Lambda(Y_M(\Omega)) & \longrightarrow & Y_M(\Omega)\\
			& T & \longmapsto & \Sigma(T) := \widetilde{\Lambda}^{-1}(T).
		\end{matrix}
	\end{equation}
	%
	%
	%
	%
	%
As equal as in the direct Calderón's mapping, we need to understand $\Sigma_M$ as a square integrable operator between separable Hilbert spaces.
} 
\section{An improved characterization of the Dirichlet-to-Neumann operator}\label{Sect:3}

\subsection{Integrability of $\Lambda_a$}

In this section we are devoted to prove integrability properties of the operators $\Lambda$ and $\Lambda_a$ for fixed $a$. First, as previously stated, we write $\Lambda_a$ for fixed $a\in X(\Omega)$ as a square integrable mapping between Hilbert spaces. Then we do the same with the mapping $a \mapsto \Lambda_a$, although this is subtle since a suitable extension is needed. {\color{black}Recall that for a probability measure $\mu$ on $H^{1/2}(\partial \Omega)$ we denote
\begin{equation}\label{W_mu}
    W_{\mu}:= L^2\Big(H^{1/2}(\partial \Omega)\times H^{1/2}(\partial \Omega),\mu\otimes\mu; \R \Big).
\end{equation}
In Lemma \ref{Lem6p1} we will show that our definition of $W_{\mu}$ is consistent with the mapping $a\mapsto\Lambda_a$. 

\begin{lemma}\label{lemma:linear-bilinear}
One has that 
\[
\ca{L}(H^{1/2}(\partial\Omega),H^{-1/2}(\partial\Omega))\hookrightarrow\ca{B}(H^{1/2}(\partial\Omega)\times H^{1/2}(\partial\Omega);\R), 
\]
with
\[
\ca{B}(H^{1/2}(\partial\Omega)\times H^{1/2}(\partial\Omega);\R)\subset L^2(H^{1/2}(\partial\Omega)\times H^{1/2}(\partial\Omega),\mu\otimes\mu;\R).
\]
\end{lemma}
\begin{proof}
	For the first inclusion, let $A\in\ca{L}(H^{1/2}(\partial\Omega),H^{-1/2}(\partial\Omega))$; we shall prove that there exists a representative $B_A\in\ca{B}(H^{1/2}(\partial\Omega)\times H^{1/2}(\partial\Omega);\R)$. For any $(f,g)\in H^{1/2}(\partial\Omega)\times H^{1/2}(\partial\Omega)$ define $B_A(f,g):=A(f)(g)$, then bilinearity is a consequence of this definition and continuity follows from the following estimate:
	\begin{align*}
		|B_A(f,g)| \le \norm{A}_{op}\norm{f}_{1/2}\norm{g}_{1/2}.
	\end{align*}
	For the second inclusion, take $B$ a bilinear form, then
	\begin{align*}
		\int_{H^{1/2}(\partial\Omega)}\int_{H^{1/2}(\partial\Omega)}|B(f,g)|^2(\mu\otimes\mu)(df,dg)\le \norm{B}_{op}^2\parent{\int_{H^{1/2}(\partial\Omega)}\norm{f}^2_{1/2}\mu(df)}^2,
	\end{align*}
	from which the claim follows. 
\end{proof}

\begin{lemma}\label{Lem6p1}
    Let $\mu$ be a probability measure on $H^{1/2}(\partial \Omega)$ with finite second moment and let $a \in X(\Omega)$. Then $\Lambda_a $ can be identified with an element of $W_{\mu}$.
\end{lemma}
\begin{proof}
     The statement is a direct consequence of Lemma \ref{lemma:linear-bilinear} since $\Lambda_a\in\ca{L}(H^{1/2}(\partial\Omega),H^{-1/2}(\partial\Omega))$.
\end{proof}
}
{\color{black}From now on we do not make a difference between $\Lambda_a$ and its representative in $W_{\mu}$}. Recall that it is always possible to find such measure $\mu$; see Lemma \ref{Measure} for further details. 

\subsection{Extension of the Calder\'on's mapping}

Recall $X(\Omega)=X_2(\Omega)$ as stated in \eqref{eq:X-L2} and $W_\mu$ as in \eqref{W_mu}.  Let $\mu$ and $\eta$ be probability measures on $H^{1/2}(\partial \Omega)$ and $H^m(\Omega)$, respectively, with finite second moment. Let $\Gamma$ be the operator
    \be\label{eq:Gamma}
        \begin{aligned}
            \Gamma: H^m(\Omega) \, &\longrightarrow \quad W_{\mu}\\
            a\quad  &\longmapsto \quad  \Gamma(a) := \begin{cases}
                \Lambda_a  & a \in X(\Omega) \cap H^m(\Omega),\\
                0  &a \notin X(\Omega) \cap H^m(\Omega).
            \end{cases}
        \end{aligned}
    \ee

{\color{black}
	
\begin{lemma}[Extension]\label{prop:L2mapping} Assume additionally that $a\in H^m(\Omega)$, with $m>\frac{d}2$. Then 
 $\Gamma$ restricted to $H^m(\Omega)$ is a $L^2(H^m(\Omega),\eta;W_{\mu})$ mapping.
\end{lemma}
\begin{proof}
   We prove that
    \be\label{eq:L2inequality}
      \int_{H^m(\Omega)} \norm{\Gamma(a)}^2_{W_{\mu}}\eta(da)<+\infty.
    \ee
    Indeed, 
    \begin{align*}
    	&\int_{H^m(\Omega)\cap X(\Omega)}\int_{H^{1/2}(\partial\Omega)}\int_{H^{1/2}(\partial\Omega)} |\prom{\Lambda_af,g}|^2 (\mu\otimes\mu)(df,dg)\eta(da) \\
    	&\le C\parent{\int_{H^{m}} \norm{a}_{\infty}^2 \eta(da)} \parent{\int_{H^{1/2}(\partial\Omega)} \norm{f}_{1/2}^2 \mu(df) }^2\\
    	&\le C\parent{\int_{H^{m}(\Omega)} \norm{a}_{H^{m}(\Omega)}^2 \eta(da)} \parent{\int_{H^{1/2}(\partial\Omega)} \norm{f}_{1/2}^2 \mu(df) }^2 < \infty,
    \end{align*}
    where we have, once again, used estimate \eqref{Lambda-f-bound} and Sobolev embedding given that $m>\frac{d}{2}$.
\end{proof}
}

\begin{remark}
Lemma \ref{prop:L2mapping} reveals that for any suitable measure $\eta$, the linear norm of the nonlinear Calder\'on{\color{black}'s} mapping $a \mapsto \Lambda_a$ is square integrable; this property is key to the proof of approximation via DeepONets.
\end{remark}

\begin{remark}
This is the unique location in the proof of Theorem \ref{MT2} that one uses more regularity on the conductivity $a$. This is needed to place the conductivity in a Hilbert space context compatible in $L^\infty(\Omega)$. We believe that this condition can be weakened provided better integrability estimates on the nonlinear Calder\'on's mapping are obtained.
\end{remark}

\subsection{Separability and choice of orthonormal basis}

{\color{black}Now, we turn to the question whether is possible to find an orthonormal basis in the complicated space
\[
W_{\mu}=L^2(H^{1/2}(\partial\Omega)\times H^{1/2}(\partial\Omega),\mu\otimes\mu;\R).
\] 
First, note that the push-forward measure $\eta\circ\Lambda^{-1}$ in $W_{\mu}$ has finite second moment. This follows from elementary properties of the push-forward of a measure, estimate \eqref{Lambda-f-bound} and Sobolev's embedding:
	\begin{align*}
		\int_{W_{\mu}} \norm{y}_{W_{\mu}}^2 (\eta\circ\Lambda^{-1})(dy)\le C\parent{\int_{H^m(\Omega)}\norm{a}^2_{H^{m}(\Omega)}\eta(da)}\parent{ \int_{H^{1/2}(\partial\Omega)} \norm{f}_{1/2}^2\mu(df)} < \infty.
	\end{align*}
Recall that both $\mu$ and $\eta$ have finite second moment in their respective spaces. Now, in order to obtain a basis for the Hilbert space at hand, we use the fact that since $\eta\circ\Lambda^{-1}$ is a finite second moment probability measure, there exists a self-adjoint, non-negative, trace class operator $C\colon W_{\mu}\to W_{\mu}$. Finally, define the basis in $W_{\mu}$ as the (orthonormal) eigenvectors of such operator.	
}

{\color{black}
	
	\subsection{Extension of the inverse Calderón's mapping}\label{ssec:6.4}
	For $M \in (0,\infty)$ recall $Y_M(\Omega)$ as stated in previous Section and $W_{\mu}$ as in \eqref{W_mu}.
	
	\medskip
	{\color{black} First of all, recall the inverse operator $\Sigma_M$ defined in \eqref{eq:inverse_operator}. Additionally, the space $Y_M(\Omega)$ is a subset of $L^{2}(\Omega)$, and then, $\Sigma_M$ can be seen as an operator from $\Lambda(Y_M(\Omega))$ into $L^{2}(\Omega)$.
		
		\medskip
		Our main objective is to have a characterization of the operator $\Sigma_M$ between two Hilbert spaces. To do so, we will provide two different approaches for the extension of $\Sigma_M$. A first insight is to extend $\Sigma_M$ into $L^2(H^{1/2}(\partial \Omega),\mu;H^{-1/2}(\partial \Omega))$, which is a Hilbert space that contains $\mathcal L(H^{1/2}(\partial \Omega),H^{-1/2}(\partial \Omega))$. The main problem is that we do not know if $L^2(H^{1/2}(\partial \Omega),\mu;H^{-1/2}(\partial \Omega))$ is a separable Hilbert space. So, for this approach we need an additional assumption on the measure $\mu$.
		
		\begin{assumptions}\label{sup:1}
			There exists a finite measure $\mu$ over $H^{1/2}(\partial \Omega)$, with finite second moment, such that the space $V_{\mu} := L^2(H^{1/2}(\partial \Omega),\mu;H^{-1/2}(\partial \Omega))$ is a separable Hilbert space.
		\end{assumptions}
		
		
		Then, for the first approach, the following Lemma can be stated:
		
		\begin{lemma}\label{lemma:L2_inv_sup1}
			Suppose that $\mu$ is a finite measure on $H^{1/2}(\partial \Omega)$ such that satisfies assumption \ref{sup:1}. Let $\nu$ be a probability measure on $V_{\mu} = L^2(H^{1/2}(\partial \Omega),\mu;H^{-1/2}(\partial\Omega))$ with finite second moment and let $M \in (0,\infty)$. Then the operator $\widetilde{\Sigma}_M$, defined as
			\be\label{eq:ext_invmap_sup1}
			\begin{aligned}
				\widetilde{\Sigma}_M: V_{\mu} \, &\longrightarrow \quad L^{2}(\Omega)\\
				T\quad  &\longmapsto \quad  \widetilde{\Sigma}_M(T) := \begin{cases}
					\Sigma_M(T)  & T \in \Lambda(Y_M(\Omega)),\\
					0  & T \notin \Lambda(Y_M(\Omega)),
				\end{cases}
			\end{aligned}
			\ee
			is a $L^2\Big(V_{\mu},\nu;L^2(\Omega)\Big)$ operator.
		\end{lemma}
		\begin{proof}
			Let $M\in(0,\infty)$. We want to prove that
			\[
			\int_{V_{\mu}} \norm{\widetilde{\Sigma}_M(T)}_{L^2(\Omega)}^2 \nu(dT) < \infty.
			\]
			From the definition of $\widetilde{\Sigma}_M$ in \eqref{eq:ext_invmap_sup1} we have that 
			\[
			\int_{V_{\mu}} \norm{\widetilde{\Sigma}_M(T)}_{L^2(\Omega)}^2 \nu(dT) = \int_{\Lambda(Y_M(\Omega))} \norm{\Sigma_M(T)}^2_{L^2(\Omega)} \nu(dT).
			\]
			Notice that for every $T \in \Lambda(Y_M(\Omega))$, there exists an unique $a_T \in Y_M(\Omega)$ such that $\Sigma_M(T)=a_T$. Therefore
			\[
			\begin{aligned}
				\int_{\Lambda(Y_M(\Omega))} \norm{\Sigma_M(T)}^2_{L^2(\Omega)} \nu(dT) &= \int_{\Lambda(Y_M(\Omega))} \norm{a_T}_{L^2(\Omega)}^2 \nu(dT)\\
				&\leq M^2 \nu\Big(\Lambda(Y_M(\Omega))\Big).
			\end{aligned}
			\] 
			Recall that $\nu$ is a probability measure. Then
			\[
			\int_{V_{\mu}} \norm{\widetilde{\Sigma}_M(T)}^2_{L^2(\Omega)} \nu(dT) \leq M < \infty.
			\]
			Concluding that $\widetilde{\Sigma}_M$ is a $L^2(V_{\mu},\nu;L^2(\Omega))$ operator for all $M \in (0,\infty)$.
		\end{proof}
		
		\begin{remark}
			As we do not know if assumption \ref{sup:1} is non empty, we will propose a second approach. If there exists that measure $\mu$, from this point we will use that measure for the rest of the paper. 
		\end{remark}
		
		For the second extension proposed, we consider the space of bilinear continuous operators from $H^{1/2}(\partial \Omega) \times H^{1/2}(\partial \Omega)$ into $\R$, namely $\mathcal{B}(H^{1/2}(\partial \Omega)\times H^{1/2}(\partial\Omega),\R)$. Notice that for any bilinear continuous form $B \in \mathcal{B}(H^{1/2}(\partial \Omega)\times H^{1/2}(\partial \Omega),\R)$, we can define a representation of $B$ in the space of continuous linear operators from $H^{1/2}(\partial \Omega)$ into $H^{-1/2}(\partial \Omega)$, in other words, for any $B \in \mathcal{B}(H^{1/2}(\partial \Omega)\times H^{1/2}(\partial \Omega),\R)$ we can define $T_B$ as follows
		\[
		\begin{matrix}
			T_B : & H^{1/2}(\partial \Omega) &\longrightarrow &H^{-1/2}(\partial \Omega)  \\
			& f & \longmapsto & T_B(f) = B(f,\cdot).
		\end{matrix}
		\]
		
		\begin{remark}
			The representation of $B$ in the space $\mathcal{L}(H^{1/2}(\partial \Omega),H^{-1/2}(\partial \Omega))$ is not unique, but we only need one of them, and we will work with the representation defined as above.
		\end{remark}

		Now let $D$ be the following subset
		\[
		D := \left\{ B \in \mathcal B(H^{1/2}(\partial\Omega) \times H^{1/2}(\partial\Omega),\R) : \> \exists a \in Y_M(\Omega), \Lambda_a \in \Lambda(Y_M(\Omega)) \hbox{ s.t. } T_B \equiv \Lambda_a \right\}.
		\]
		
		In other words, $D$ is the set of all the bilinear continuous operators such that their representation on $\mathcal L (H^{1/2}(\partial \Omega),H^{-1/2}(\partial \Omega))$ is exactly a Dirichlet-to-Neumann operator. The subset $D$ allow us to well-define the functional
		
		\[
		\begin{matrix}
			\chi : & D &\longrightarrow & \Lambda(Y_M(\Omega))  \\
			& B & \longmapsto & \chi(B) := T_B.
		\end{matrix}
		\]
		
		\begin{remark}
			Recall that $\chi$ is the functional such that for any bilinear continuous operator $B$ on $D$, $\chi(B)$ gives the Dirichlet-to-Neumann operator associated.
		\end{remark}
		
		For all $M \in (0,+\infty)$, the composition of $\Sigma_M$ with $\chi$ give us an operator such that represents the inverse Calderón's mapping between the spaces $D$ and $L^2(\Omega)$, and then we can extend it to $W_{\mu}$ from the fact that
		\[
		D \subset \mathcal{B}(H^{1/2}(\partial\Omega) \times H^{1/2}(\partial \Omega),\R) \subset W_{\mu}.
		\]
		In resume, we define the operator
		\[
		\begin{matrix}
			\pi_M : & D &\longrightarrow & L^2(\Omega)  \\
			& B & \longmapsto & \pi_M(B) := \Sigma_M \circ \chi (B).
		\end{matrix}
		\]
		And this operator can be extended in the following way
		\begin{equation}\label{eq:inverse_operator_BL}
			\begin{aligned}
				\widetilde{\pi}_M: W_{\mu} \, &\longrightarrow \quad L^{2}(\Omega)\\
				B \quad  &\longmapsto \quad  \widetilde{\pi}_M(B) := \begin{cases}
					\pi_M(B)  & B \in D \cap W_{\mu},\\
					0  & B \notin D \cap W_{\mu}.
				\end{cases}
			\end{aligned}
		\end{equation}
		\begin{remark}
			$\widetilde{\pi}_M$ is an operator such that, if $B$ is a bilinear continuous map that has a linear representation which is a Dirichlet-to-Neumann map $\Lambda_{a_B}$ for some $a_B \in Y_M(\Omega)$, then $\widetilde{\pi}_M(B) = a_B$. Therefore $\widetilde{\pi}_M$ is indeed a characterization of the inverse Calderón's mapping on the space $W_{\mu}$. 
		\end{remark}
		The next Lemma can be stated for $\widetilde{\pi}_M$, for any $M \in (0,+\infty)$:
		\begin{lemma}\label{lemma:L2inv_BL}
			Let $\nu$ be a probability measure on $W_{\mu}$ with finite second moment and let $M \in (0,\infty)$. Then $\widetilde{\pi}_M$, defined as in \eqref{eq:inverse_operator_BL}, is an $L^2(W_{\mu},\nu;L^2(\Omega))$ operator.
		\end{lemma}
		
		\begin{proof}
			We want to prove that
			\[
			\int_{W_{\mu}} \norm{\widetilde{\pi}_M(B)}^{2}_{L^{2}(\Omega)} \nu(dB) < \infty.
			\]
			From the definition of $\widetilde{\pi}_M$ in \eqref{eq:inverse_operator_BL} we have that
			\[
			\int_{W_{\mu}} \norm{\widetilde{\pi}_M(B)}^{2}_{L^{2}(\Omega)} \nu(dB) = \int_D \norm{\pi_M(B)}^{2}_{L^{2}(\Omega)}\nu(dB).
			\]
			Notice that the results on $D$ and $\pi_M$, implies that for any $B \in D$ there exists $a_B \in Y_M(\Omega)$ and $\Lambda_{a_B} \in \Lambda(Y_M(\Omega))$ such that 
			\[
			\pi_M(B) = \Sigma_M \circ \chi(B) = \Sigma_M(\Lambda_{a_B}),		
			\]
			and recall that $\Sigma_M$ is the inverse operator of $\Lambda_{\cdot}$, in the space $Y_M(\Omega)$. Then $\Sigma_M(\Lambda_{a_B}) = a_B$. Now, $a_B \in Y_M(\Omega)$ implies that $a_B \leq M$, therefore
			\[
			\begin{aligned}
				\int_D \norm{\pi_M(B)}^{2}_{L^{2}(\Omega)}\nu(dB) &= \int_{D} \norm{a_B}_{L^2(\Omega)}^2 \nu(dB) \\
				&\leq M^2 \nu(D).
			\end{aligned}
			\]
			From the fact that $\nu$ is a probability measure on $W_{\mu}$, it follows that
			\[
			\int_{W_{\mu}} \norm{\widetilde{\pi}_M(B)}^{2}_{L^{2}(\Omega)} \nu(dB) \leq M^2 < \infty.
			\]
			Concluding that $\widetilde{\pi}_M$ is a $L^2(W_{\mu},\nu;L^2(\Omega))$ operator for all $M \in (0,\infty)$.
		\end{proof}
		
		\begin{remark}
			Lemma \ref{lemma:L2inv_BL} can be stated without Assumption \ref{sup:1}. This implies that we only need that $\mu$ is a finite measure on $H^{1/2}(\partial \Omega)$ with finite second moment.
		\end{remark}
	}	
}

\section{Description of required DeepONets}\label{Sect:4}

\subsection{Extension to measurable mappings in Hilbert spaces}

The previous and original conception for DeepONets is placed in the general Banach setting. Sometimes one needs less generality but more precision in the required approximations, for instance, being placed in Hilbert spaces of (not necessarily) continuous mappings. This is the case of the Calder\'on's problem, that we will see can be nicely put in a Hilbert space context. We first state the functional setting required for this construction.

\medskip

{\it
\noindent
{\bf General setting}. In our case, we will seek for a general Deep Learning framework that approximates nonlinear mappings $F\colon H\to W$, with $H$ and $W$ general separable Hilbert spaces suitable for the Calder\'on's problem. 
}

\medskip

There is an easy solution to this problem: just provide suitable finite dimensional approximation spaces for $H$ and $W$, and then project and extend the input to obtain a desired output. This is usually the case in Linear Algebra applications. However, since we are dealing with mappings with deep encoded properties, and we require methods that efficiently approximate functions (such as standard DNNs), specially in higher dimensions, this approach fails to provide all the required characteristics and conditions of our problem. 

\medskip

Approximating $F$ by a neural network $F^{\theta}\colon H\to W$, where $\theta$ is a finite dimensional parameter, is precisely the main objective of Hilbert-valued DeepONets. One can compare these ideas with the approximation of measurable functions with simple functions in integral-type distance and continuous functions with polynomials in uniform norm.

\medskip

In a previous work by one of us \cite{Cas22}, by using ideas from \cite{LMK21} we extend the previous property from \cite[Lemma $7$]{CC95} to any compact set $V$ in a general Hilbert space $H$. Here, the considered transformation is the projection onto a finite dimensional subspace of $H$. It was proved a general result that can be rephrased as follows:

\begin{theorem}\label{vago}
	 Any $L^2(H,\mu;W)$ mapping $G$ (with $H$ endowed with any finite measure $\mu$ with finite second momentum) can be approximated by a DeepONet $F^{\theta}: H\to W$, $F^{\theta} \in L^2(H,\mu;W)$.
\end{theorem}

For a precise description of this result, see Theorem \ref{theorem:infiniteapprox} or \cite{Cas22}. The parameter $\theta$ is finite dimensional, and essentially represents three independent objects: a finite dimensional projection space of dimension $\tilde d$, a finite DNN of parameter $\theta_{fin}$, and finally an extension parameter $m$ representing a finite dimensional subspace of $W$. Theorem \ref{vago} has impressive consequences, being the most important a clear approximation scheme for the Calder\'on's mapping, not only for fixed conductivity $a$, but also in the much general setting where $a$ ranges on $L^\infty$ (plus some additional assumptions stated in Section \ref{Sect:1b}). As a consequence, we advance that the conductivity $a$ will be part of a DNN scheme based in the infinite-dimensional DeepONet machinery (and not only the standard DNN procedure applied to Inverse problems).

\subsection{How to measure errors in infinite dimensions. The probabilistic approach}

{\color{black} Clearly the objects arising from the Calder\'on's problem are nonlinear 
mappings $F$ between infinite dimensional spaces. The typical starting point here is a mapping $F\colon H\to W$ enjoying some property related to the Calder\'on's problem and a suitable class of functions (model) $\set{F^{\theta}\colon H\to W}_{\theta\in\R^{\kappa}}$ for some $\kappa\in\bb{N}$. The aim is to minimize the mean square error (MSE)
\begin{align}\label{error-measure}
	\bb{E}_{x\sim\mu}\norm{F(x)-F^{\theta}(x)}_{W}^2 = \int_{H} \norm{F(x)-F^{\theta}(x)}_{W}^2 \mu(dx),
\end{align}
where $\mu$ is a borelian probability measure on $H$. This is a straightforward extension of what can be found in the machine learning literature where one has access to a data set $\set{(x_i,y_i)}_{i=1}^N$, such that $F(x_i) = y_i$ for all $i=1,...,N$. By defining
\begin{align*}
	\mu = \frac{1}{N}\sum_{i=1}^N \delta_{x_i},
\end{align*}
one arrives at the typical MSE as presented in the machine learning literature. Later on, we will require $\mu$ to have finite second moment, and a natural question to ask is whether or not such a measure exists}. Indeed, one has the following: 

\begin{lemma}\label{Measure}
	For any separable Hilbert space $H$ there exists at least one probability measure $\mu$ defined on $H$ with finite second moment.
\end{lemma}


\begin{proof}[Sketch of proof] The interested reader can consult e.g. the monograph \cite{Par05}, Chapter VI. The construction goes as follows: let $S$ be any positive semidefinite Hermitian operator defined on $H$ with the property that for some orthonormal basis $\{e_i\}_i$ on $H$, $\sum_{i\geq 1} \langle Se_i, e_i\rangle < \infty$.
Then the function $\varphi(y) = \exp\left(-\frac{1}{2} \langle Sy,y\rangle\right)$, is the characteristic function of a probability measure $\mu$ on $H$. Moreover, $\mu$ satisfies
\begin{equation}\label{2do momento}
\int_{H} \norm{x}^2_{H} \mu(dx) < \infty.
\end{equation}
This ends the proof.
\end{proof}
%

\subsection{Approximation of infinite dimensional mappings via DeepONets}

Now we are ready to start the proof of Theorems \ref{MT1} and \ref{MT2}. As naturally expected, we first prove the approximation result for the Dirichlet-to-Neumann map, and then for the Calder\'on's mapping. The main ingredient for the proof will  the following universal approximation result.

{

\begin{theorem}
\label{theorem:infiniteapprox}
	 Let $G\colon H\to W$ be a $L^2(H,\mu;W)$ mapping with $H$ endowed with the measure $\mu$. Then, for any $\varepsilon>0$ there exist a set of parameters $(d,\theta,m)\in\ca{N}_{\sigma,L}^{H\to W}$ and a DeepONet $F^{d,\theta,m}: H\to W$, $F^{d,\theta,m} \in L^2(H,\mu;W)$ such that,
	\begin{align}\label{MTaprox}
		\int_{H}\norm{G(x)-F^{d,\theta,m}(x)}_W^2\mu(dx) <\varepsilon.
	\end{align} 	
\end{theorem}

\begin{remark}
The proof of this result is essentially contained in \cite[Theorem $5.14$]{Cas22}. For the sake of completeness we include below a sketch of proof.  For further details, see the previous reference. 
\end{remark}

\begin{remark}
Estimate \eqref{MTaprox} is deeply motivated by a result by Tianping Chen and Hong Chen \cite{CC95}, who showed the universal approximation to nonlinear operators by neural networks with arbitrary activation functions and its applications to dynamic systems. 
\end{remark}

\begin{remark}
The proof of this result one can estimate the number of hidden layers $L$ in $F^{d,\theta,m}$. Indeed, in \cite{Cas22} one can take $L=7$. This will be used in our main results as well. 
\end{remark}

\begin{proof}[Sketch of proof of Theorem \ref{theorem:infiniteapprox}]
 Fix $\varepsilon > 0$. Here $d$ represents the size of a DNN, and not the dimension of the Elliptic Calder\'on's problem stated in the introduction.

\medskip

{\bf Step 1. Case $G$ of finite range.} First we apply \cite[Theorem $4$]{CC95} to prove the existence of a DeepONet $F^{d,\theta_1,m}=F^{H,d,\theta_1,m,\R^m}$ with $(d,\theta_1,m)\in\ca{N}^{H\to\R^{m}}_{\sigma,2}$, approximating $G$ in the case 
\[
\hbox{$G\colon K\subseteq H\to\R^m$, \quad $G\in C(K;\R^m)$, \quad  for fixed $m\in\bb{N}$ and $K$ a compact in $H$.}
\]
As expressed before, this is the first result proved by Chen and Chen \cite{CC95}. Notice that the number of hidden layers here can be chosen to be $L=2$.
 

\medskip

{\bf Step 2.} By approximations via dominated convergence we can reduce the problem to the case of a bounded mapping $G\in L^2(H,\mu;W)$ such that 
\[
\|G\|_{L^\infty(H,W)}= \hbox{esssup}_{x\in H} \|G(x)\|_{W} <+\infty.
\]
This is done by considering the modified mapping defined by $G(x)$ if $\|G(x)\|_W\leq R$, and $ \frac{RG(x)}{\|G(x)\|_W}$ if $\|G(x)\|_W\geq R$, with $R\to +\infty$.  Now, given that $G$ is a measurable function and $\mu$ a finite measure, by virtue of the Lusin's theorem,  there exists a compact $K=K_{\varepsilon}\subseteq H$ such that the restriction of $G$ to $K$ is continuous and $\mu(H\setminus K)\le\varepsilon$. 

\medskip

We now use the properties \cite{Cas22} of the compact set $G(K)$ embedded in the Hilbert space $W$ to get a {\bf fixed} $\kappa=\kappa_{\varepsilon}\in\bb{N}$ such that,
\begin{align}\label{Proyeccioness}
    \underset{x\in K}{\sup}\norm{G(x)- P_\kappa(G(x))}_W <\varepsilon, \qquad P_\kappa(w):=\sum_{i=1}^{\kappa} \prom{w,g_i}g_i.
\end{align}
Note that $P_{\kappa}$ is not the same operator defined in \eqref{EE}.This means that we can work, up to an $\varepsilon$-error, with {\it the orthogonal projection} of $w\in G(K)$ onto a finite dimensional space. 
Let $\widetilde{G}=P_{\kappa}\circ G\colon K\to W$. Note that $\widetilde G$ has finite range, and from \eqref{Proyeccioness},
		\begin{align*}
			\underset{x\in K}{\sup}\norm{G(x) - \widetilde{G}(x)}_W<\varepsilon.
		\end{align*}
Thus we can use Step $1$ to get a first DeepOnet $F^{d,\theta_1,\kappa}=F^{H,d,\theta_1,\kappa,\R^\kappa}$ with $(d,\theta_1,\kappa)\in\ca{N}^{H\to\R^{\kappa}}_{\sigma,2}$ that is $\varepsilon$-good approximating the finite dimensional representation of $G$ on $K$ and is also bounded.

\medskip



Indeed, one can first obtain that for the continuous function $P_{W,\kappa}\circ\widetilde{G}\colon K\to W\to \R^{\kappa}$, we can take $(d,\theta_1,\kappa)\in\ca{N}^{H\to\R^{\kappa}}_{\sigma,2}$ such that,
		\begin{align*}
			\underset{x\in K}{\sup} \norm{ F^{H,d,\theta_1,\kappa,\R^{\kappa}}(x) - (P_{W,\kappa}\circ \widetilde{G})(x)}_{\R^\kappa} < \delta.
		\end{align*}
Now take any $x\in K$ and the DeepONet generated by $(H,d,\theta_1,\kappa,W)$ to obtain
		\begin{align}\label{eq:bound1_uat}
			\norm{F^{H,d,\theta_1,\kappa,W}(x) - \widetilde{G}(x)}_W &= \norm{({E}_{W,\kappa}\circ f^{\theta_1}\circ P_{H,d})(x) - \widetilde{G}(x)}_W\nonumber\\
			& = \norm{\sum_{i=1}^{\kappa} (f^{\theta_1}\circ P_{H,d})(x)_i g_i - \sum_{i=1}^{\kappa} \prom{G(x),g_i}_W g_i }_W\nonumber\\
			&= \norm{(f^{\theta_1}\circ P_{H,d})(x) - \parent{\prom{G(x),g_i}_W}_{i=1}^{\kappa}}_{\R^{\kappa}}\nonumber\\
			&= \norm{F^{H,d,\theta_1,\kappa,\R^{\kappa}}(x) - (P_{H,\kappa}\circ \widetilde{G})(x)}_{\R^{\kappa}} < \delta.
		\end{align}
		Therefore, by using previous estimate and that $G$ is bounded by some $M>0$, one show that $F^{H,d,\theta_1,\kappa,W}$ is also bounded. Indeed,
		\begin{align*}
			\norm{F^{H,d,\theta_1,\kappa,W}(x)}_W \le \norm{F^{H,d,\theta_1,\kappa,W}(x) - \widetilde{G}(x)}_W + \norm{\widetilde{G}(x) - G(x)}_W + \norm{G(x)}_W < 2\delta + M.
		\end{align*}
		\textbf{Step 4.}
		Applying the clipping Lemma \cite{Cas22}, one can assume $\delta$ small enough such that one can find  $\theta_2\in\ca{N}_{\sigma,5,\kappa,\kappa}$ such that,
		\begin{align}\label{eq:clipping-gamma}
			\begin{dcases}
				\norm{f^{\theta_2}(x) - x} < \delta,\ &\norm{x} < M + 2\delta\\
				\norm{f^{\theta_2}(x)} \le 2M,\ &\forall x\in \bb{R}^{\kappa}.
			\end{dcases}
		\end{align}
		Recall that the norm used in previous equation is the usual norm in $\R^{\kappa}$, and $\theta_2$ is constructed using the classical activation function $\sigma=\sigma_{\text{ReLu}}$. Consider the following function:
		\begin{align*}
			E_{W,\kappa}\circ f^{\theta_2}\circ f^{\theta_1}\circ P_{H,d} = E_{W,\kappa}\circ f^{\theta_2\circ\theta_1}\circ P_{H,d}=F^{H,d,\theta_1\circ\theta_2,\kappa,W}.
		\end{align*}
We have
		\begin{align*}
			\norm{F^{H,d,\theta_2\circ\theta_1,\kappa,W}(x) - \widetilde{G}(x)}_W &\le \norm{F^{H,d,\theta_2\circ\theta_1,\kappa,W}(x) - F^{H,d,\theta_1,\kappa,W}(x)}_W + \norm{F^{H,d,\theta_1,\kappa,W}(x) - \widetilde{G}(x)}_W\\
			&\le \norm{\sum_{i=1}^{\kappa}f^{\theta_2}_i\parent{f^{\theta_1}\parent{P_{H,d}(x)}} g_i -  \sum_{i=1}^{\kappa} \parent{f^{\theta_1}\circ P_{H,d}}_i (x) g_i}_W + \delta\\
			&\le\norm{f^{\theta_2}\parent{f^{\theta_1}\parent{P_{H,d}(x)}} - f^{\theta_1}\parent{P_{H,d}(x)}}_{\R^{\kappa}} + \delta < 2\delta,
		\end{align*}
		where we used estimates \eqref{eq:bound1_uat} and \eqref{eq:clipping-gamma}.\\
		
		\textbf{Step 5.} Now we combine all previous bounds, let $F=F^{H,d,\theta_2\circ\theta_1,\kappa,W}$ with $(d,\theta_2\circ\theta_1,\kappa)\in\set{d}\times\ca{N}_{\sigma,7,d,\kappa}\times\set{\kappa}$, then 
		\begin{align*}
			\int_H \norm{G(x) - F(x)}_W^2 \mu (dx) &= \int_{H\setminus K} \norm{G(x) - F(x)}_W^2 \mu (dx) + \int_{K} \norm{G(x) - F(x)}_W^2 \mu (dx)\\
			&\le 2\int_{H\setminus K} \norm{G(x)}_W^2\mu(dx) + 2\int_{H\setminus K}\norm{F(x)}^2_W\mu(dx) \\
			&~{} \quad  + 2\int_{K} \norm{G(x) - \widetilde{G}(x)}_W^2\mu(dx) + 2\int_{K} \norm{\widetilde{G}(x) -  F(x)}_W^2\mu (dx)\\
			&\le \mu\parent{H\setminus K}(2M^2 + 2M^2) + 2\delta^2 + 2\delta^2 \le 8\delta^2=\varepsilon,
		\end{align*} 
		which is the desired conclusion.
\end{proof}

\section{Approximation of the   Dirichlet-to-Neumann operator using DeepONets}\label{Sect:5}

In this section we prove the main results of this paper, that is, Theorems \ref{MT1}, \ref{MT2} {\color{black} and \ref{MT3}}. We first start with the proof of Theorem \ref{MT1}. Recall that for this proof we only need $a\in X(\Omega).$

\subsection{Approximation of the   Dirichlet-to-Neumann map $\Lambda_a$ for fixed $a$}

In what follows, we assume $a\in X(\Omega)$ fixed. First, we prove and approximation result in the case of a probability measure defined on $H^{1/2}(\partial \Omega).$ This result is a consequence of Theorem \ref{theorem:infiniteapprox}.

\begin{theorem}\label{MT1-general-case}
    Let $\mu$ be a finite measure on $H^{1/2}(\partial\Omega)$ with finite second moment and $a\in X(\Omega)$. Let $\Lambda_a$ be as in \eqref{weak-fomulation1}. Then, there exist $(d,\theta,m)$ and a Deep-H-Onet $F^{d,\theta,m}: H^{1/2}(\partial \Omega) \longrightarrow H^{-1/2}(\partial \Omega)$ as in \eqref{eq:DO-def}, such that $F^{d,\theta,m}\in L^2( H^{1/2}(\partial \Omega),\mu;  H^{-1/2}(\partial \Omega))$ and 
	\begin{equation}\label{eq:DO_aprox}
		\int_{H^{1/2}(\partial \Omega)} \norm{\Lambda_a f - F^{d,\theta,m}f}_{-1/2}^2 \mu(df) < \varepsilon.
	\end{equation}
\end{theorem}
\begin{proof}
    Recall that from Lemma \ref{lemma:Lambda-linear-bounded} one has $\Lambda_a \in  \mathcal L(H^{1/2}(\partial\Omega),H^{-1/2}(\partial\Omega))$. By Theorem \ref{theorem:infiniteapprox} applied to $H=H^{1/2}(\partial\Omega)$ and $W=H^{-1/2}(\partial\Omega)$, {\color{black}it is sufficient to show the square-integrability of $\Lambda_a$}. Indeed,
    \begin{align*}
    	\int_{H^{1/2}(\partial\Omega)} \norm{\Lambda_a f}_{1/2}^2 \mu(df) \le \norm{a}_{\infty}\int_{H^{1/2}(\partial\Omega)}\norm{f}^2_{1/2}\mu(df),
    \end{align*}
	which is finite since $\mu$ has finite second moment and $a$ is bounded.
\end{proof}

{\color{black}

\subsection{An alternative method for the proof of Theorem \ref{MT1-general-case}}
	
In the following we will comment on a alternative proof of Theorem \ref{MT1-general-case} taking advantage of $\Lambda_a$ being linear and performing the error decomposition proposed in \cite{LMK21,BHK+21}. The reader will see that with these techniques one can get a better grasp on the error incurred by the dimension reduction and by the finite dimensional neural network. First, let us introduce some notation. Let $\varphi$ be the isometric isomorphism between $H^{-1/2}(\partial\Omega)$ and $H^{1/2}(\partial\Omega)$. We argue that it is enough to approximate the mapping $\Lambda_a^{\varphi}:=\varphi\circ\Lambda_a$ by a DeepOnet $F^{\theta}\colon H^{1/2}(\partial\Omega)\to H^{1/2}(\partial\Omega)$, it is key that $\varphi$ is an isometry. Indeed, we aim to minimize
\begin{align}\label{eq:lamda-varphi}
	\bb{E}_{f\sim\mu}\norm{(\varphi^{-1}\circ F^{\theta})(f)-\Lambda_a(f)}^2_{-1/2} = \bb{E}_{f\sim\mu}\norm{F^{\theta}(f)-\Lambda^{\varphi}_a(f)}^2_{1/2}.
\end{align}	
The simplification that has been introduced along with $\varphi$ allows us to avoid the rather complicated dual space $H^{-1/2}(\partial\Omega)$. Furthermore, note that from the definition of $\Lambda_a$, one deduces that for all $f,g\in H^{1/2}(\partial\Omega)$
\begin{align*}
	\prom{\Lambda_a f,g} = \prom{\Lambda_a^{\varphi}f,g}_{1/2},
\end{align*}	
where $\prom{\cdot,\cdot}_{1/2}$ denotes the inner product in the Hilbert space $H^{1/2}(\partial\Omega)$, consequently we denote by $\norm{\cdot}_{1/2}$ the norm in $H^{1/2}(\partial\Omega)$ and $\norm{\cdot}_{-1/2}$ the norm in $H^{-1/2}(\partial\Omega)$. From that observation follows that $\Lambda_a^{\varphi}$ is a self-adjoint bounded linear operator and, therefore, we have access to its eigen-system $\set{\lambda_k,\psi_k\colon k\in\bb{N}}$ on $H^{1/2}(\partial\Omega)$. In the other hand, from the measure $\mu$ we have access to the bounded, self-adjoint and, more important, trace-class operator $Q = Q_{\mu}$ defined for $f,g\in H^{1/2}(\partial\Omega)$ by
\begin{align*}
	\prom{Qg,h}_{1/2} = \bb{E}_{f\sim\mu}\prom{f,g}_{1/2}\prom{f,h}_{1/2}.
\end{align*}
Said linear operator has an eigen-system $\set{\alpha_k,e_k\colon k\in\bb{N}}$ which satisfies
\begin{align*}
	\text{Tr}(Q) = \sum_{k=1}^{\infty} \alpha_k = \bb{E}_{f\sim\mu}\norm{f}_{1/2}^2 < \infty. 
\end{align*}
Use the eigen-system of $Q$ to define the linear operators $E_d = E_{H^{1/2}(\partial\Omega),d}$ and $P_d=P_{H^{1/2}(\partial\Omega),d}$, recall \eqref{EE}. The resulting operators are linear and such that their operator norm is bounded by $1$.

\medskip

Now, accordingly to Definition \ref{def:deeponets}, we are to write down the right hand side of \eqref{eq:lamda-varphi} with $F^{\theta}=E_d\circ f^{\theta}\circ P_d$ and apply the mentioned error decomposition (see e.g. \cite[Lemma 3.4]{LMK21}). 
\begin{align*}
&	\bb{E}_{f\sim\mu}\norm{(E_d\circ f^{\theta}\circ P_d)(f) - \Lambda_a^{\varphi}(f)}_{1/2}^{2} \\
&\le 3\bb{E}_{f\sim\mu}\norm{(E_d\circ f^{\theta}\circ P_d)(f)-(E_d\circ P_d\circ \Lambda_a^{\varphi}\circ E_d\circ P_d)(f)}_{1/2}^2\\
	& \quad + 3\bb{E}_{f\sim\mu}\norm{(E_d\circ P_d\circ \Lambda_a^{\varphi}\circ E_d\circ P_d)(f) - (E_d\circ P_d\circ \Lambda_a^{\varphi})(f)}_{1/2}^2\\
	& \quad + 3\bb{E}_{f\sim\mu}\norm{(E_d\circ P_d\circ \Lambda_a^{\varphi})(f) - \Lambda_a^{\varphi}(f)}_{1/2}^2\\
	&\quad  =: 3(I_1+ I_2 + I_3).
\end{align*}
\begin{itemize}
	\item[\textbf{$I_1$}:] By using that $E_d:\Rd\to H^{1/2}(\partial\Omega)$ satisfies $\norm{E_d}_{op}\le 1$ and the change of variable for the pushforward,
	\begin{align*}
		I_1\le \bb{E}_{x\sim\mu\circ P_d^{-1}}\norm{f^{\theta}(x) - (P_d\circ \Lambda_a^{\varphi}\circ E_d)(x)}^2_{\Rd}.
	\end{align*}
	Now, at this point one can stop and just argue that since $(P_d\circ\Lambda_a^{\varphi}\circ E_d)$ is a finite dimensional map and $\mu\circ P_d^{-1}$ is a finite measure, one could apply the classic UAT and find a parameter $\theta$ such that the previous term is small. But, we rather note that the following holds for any $x\in\Rd$,
	\begin{align}\label{eq:finite-dim-dtn}
		(P_d\circ \Lambda_a^{\varphi}\circ E_d)(x) = \begin{pmatrix} \prom{\Lambda_a^{\varphi}e_1,e_1}_{1/2} & \cdots & \prom{\Lambda_a^{\varphi}e_d,e_1}_{1/2}\\ \vdots & \ddots & \vdots \\ \prom{\Lambda_a^{\varphi}e_1,e_d}_{1/2} & \cdots & \prom{\Lambda_a^{\varphi}e_d,e_d}_{1/2} \end{pmatrix}x .
	\end{align}
	It is natural then to suppose our model be given by the linear function $f^{\theta}(x) = \theta x$ for $\theta\in\R^{d\times d}$. Follows by using that $\norm{P_d}_{op}\le 1$,
	\begin{align*}
		I_1\le \bb{E}_{f\sim\mu}\norm{f}_{1/2}^2\sum_{i,j}^d |\theta_{i,j} - \prom{\Lambda_a^{\varphi}e_j,e_i}_{1/2}|^2.
	\end{align*}
	An important observation at this point is that, if $\Lambda_a^{\varphi}$ were trace-class, should we use the eigen-system of $\Lambda_a^{\varphi}$ instead of that of $Q$ to base the operators $E_d$ and $ P_d$ on. This would give a cleaner result reducing the complexity of the optimization problem from $d^2$ to $d$ since the obtained matrix \eqref{eq:finite-dim-dtn} would be a diagonal one containing the first $d$ eigenvalues of $\Lambda_a^{\varphi}$.
	\item[\textbf{$I_2$}:] By the already mentioned arguments and \eqref{Lambda-f-bound} (recall the constant $C$ for the referenced equation),
	\begin{align*}
		I_2\le C\norm{a}_{\infty}\bb{E}_{f\sim\mu}\norm{(E_d\circ P_d)(f) - f}^2_{1/2} \le C\norm{a}_{\infty} \sum_{k={d+1}}^\infty \alpha_k,
	\end{align*}
	where we have used that $E_d\circ P_d = \sum_{k=1}^d\prom{\cdot,e_k}_{1/2}e_k$. Thus, as already mentioned in \cite{BHK+21}, the approximation error is closely related to the spectral decay of $Q$.\\
	\item [\textbf{$I_3$}:] By the arguments mentioned before,
	\begin{align*}
		I_3= \bb{E}_{f\sim\mu}\norm{(E_d\circ P_d)(\Lambda_a^{\varphi}f) - \Lambda_a^{\varphi}f}_{1/2}^2.
	\end{align*}
	Now, here is where our argument differs from \cite{BHK+21,LMK21}. In those articles, the authors project onto a finite-dimensional subspace learned from data and for which the projection error is estimated. We will follow another idea: to obtain the desired tolerance $\varepsilon$, first we prove that there exists a compact set $K$ with $\varepsilon$-almost full measure. We then apply Lemma \ref{lemma:aris} with the ONB $\set{e_k\colon k\in\bb{N}}$, choosing the compact $\Lambda_a^{\varphi}K$ and tolerance $\varepsilon$.   
\end{itemize}	  
}
%
%
%
%


\subsection{Approximation of the Calder\'on{\color{black}'s} mapping $\Lambda$}

Now we prove a slightly general case of Theorem \ref{MT2}. For brevity, recall the notation
\[
W_{\mu}:= L^2(H^{1/2}(\partial \Omega)\times H^{1/2}(\partial \Omega),\mu\otimes\mu; \R),
\]
for $\mu$ probability measure on $H^{1/2}(\partial \Omega)$ {\color{black} with finite second moment}. The proof of Theorem \ref{MT2} is an easy consequence of this corollary.

\begin{theorem} \label{MT2-general-case}
Let $\mu$ and $\eta$ be probability measures on $H^{1/2}(\partial \Omega)$ and $H^m(\Omega)$, respectively, with finite second moment. Let $\Lambda$ as in \eqref{Lambda}. Then, for all $\varepsilon>0$ there exists $(d,\theta,\kappa)$ such that
\be\label{eq:DO_Lambda}
\int_{X(\Omega)\cap H^{m}(\Omega)} \norm{\Lambda_a - F^{d,\theta,\kappa}(a)}^2_{W_{\mu}} \eta(da) < \varepsilon.
\ee
\end{theorem}

\begin{proof}
To prove inequality \eqref{eq:DO_Lambda} we will work with the mapping $\Gamma:H^{m}(\Omega)\to W_{\mu}$ defined in \eqref{eq:Gamma}. By Lemma \ref{prop:L2mapping}, $\Gamma$ is an $L^2(H^{m}(\Omega),\eta,W_{\mu})$ mapping. Now we can use Theorem \ref{theorem:infiniteapprox} to prove that for all $\varepsilon>0$ there exists $(d,\theta,\kappa) \in \mathcal{N}^{H^{m}(\Omega)\to W_{\mu}}_{\sigma,7}$ such that
\[
\int_{H^{m}(\Omega)} \norm{\Gamma(a) - F^{d,\theta,\kappa}(a)}_{W_{\mu}}^{2} \eta(da) < \varepsilon.
\]
By definition of $\Gamma$ we have
\begin{align*}
\int_{H^{m}(\Omega)} \norm{\Gamma(a) - F^{d,\theta,\kappa}(a)}_{W_{\mu}}^{2} \eta(da) &= \int_{X(\Omega)\cap H^{m}(\Omega)} \norm{\Lambda_a - F^{d,\theta,\kappa}(a)}_{W_{\mu}}^{2} \eta(da) \\
& \quad + \int_{X(\Omega)^c \cap H^{m}(\Omega)} \norm{F^{d,\theta,\kappa}(a)}_{W_{\mu}}^{2} \eta(da)\\
&\geq \int_{X(\Omega)\cap H^{m}(\Omega)} \norm{\Lambda_a - F^{d,\theta,\kappa}(a)}_{W_{\mu}}^{2} \eta(da).
\end{align*}
Therefore, for all $\varepsilon>0$ there exists $(d,\theta,\kappa) \in \mathcal{N}_{\sigma,7}^{H^{m}(\Omega) \to W_{\mu}}$ such that
\[
    \int_{X(\Omega)\cap H^{m}(\Omega)} \norm{\Lambda_a - F^{d,\theta,\kappa}(a)}^2_{W_{\mu}} \eta(da) < \varepsilon.
\]
This ends the proof.
\end{proof}

{\color{black}
\begin{remark}
We remark the following:
\begin{itemize} 
\item As stated before, we need more regularity on the conductivity $a$ in Theorem \ref{MT2-general-case} essentially to show that the Calder\'on's mapping can be embedded in a suitable square-integrable Hilbert space, see Lemma \ref{prop:L2mapping}.
\item For this particular case one could take advantage of the structure of $\Lambda$. Indeed, we could for example bypass the extension by 0 requiring the measure $\eta$ to be supported on $X(\Omega)$ since the operator is not defined outside this set. To do so, we just need to change the arguments in Step 2 of the proof of Theorem \ref{theorem:infiniteapprox}. There and only there we introduce the trace Borelian measure space $(X(\Omega)\cap H^{m}(\Omega),\eta)$ and apply Luzin's theorem to obtain a compact set $K\subset X(\Omega)\cap H^{m}(\Omega)$ with $\eta(K)\ge 1-\varepsilon$. Furthermore, $\Lambda\big|_{K}$ is continuous. This change would ensure that the obtained compact set is a subset of $X(\Omega)\cap H^{m}(\Omega)$ and therefore $\Lambda$ continue to be well-defined. The next steps can remain unchanged.  
\end{itemize}
\end{remark}
}

{\color{black}
	
	\subsection{Approximation of the inverse Calderón's mapping}
	
	In this subsection we prove the approximation result for the inverse Calder\'on's mapping announced in Main results. {\color{black} We also provide two Theorems, each one related with each proposed extension.\\
		
		The first Theorem we prove is for the extension of the inverse Calderón's mapping $\Sigma_M$ on the space $V_{\mu} = L^2(H^{1/2}(\partial \Omega),\mu,H^{-1/2}(\partial \Omega))$. This Theorem states that there exists a DeepOnet that if $a \in L^2(\Omega)$ is a conductivity such that 
		\[
		\frac{1}{M} \leq a \leq M,
		\]
		then it can be approximated by the previous mentioned DeepOnet evaluated on the Dirichlet-to-Neumann operator $\Lambda_a$.}
	
	{\color{black}
		\begin{theorem}\label{teo:DO_sup1}
			Let $\mu$ and $\nu$ be probability measures on $H^{1/2}(\partial \Omega)$ and $V_{\mu}$, respectively, with finite second moment. Suppose that $\mu$ satisfies assumptions \ref{sup:1}. For $M \in (0,\infty)$ let $\Sigma_M$ be the inverse Calderón's mapping restricted to $Y_M(\Omega)$, defined as in \eqref{eq:inverse_operator}. Then, for all $\varepsilon > 0$ there exists $(d,\theta,\kappa)$ such that
			\begin{equation}\label{teo:inv_deeponet_sup1}
				\int_{\Lambda(Y_M(\Omega)) \cap V_{\mu}} \norm{\Sigma_M(T) - F^{d,\theta,\kappa}(T)}^2 \nu(dT) < \varepsilon.
			\end{equation}
		\end{theorem}
		
		\begin{proof}
			Let $M \in (0,\infty)$. As equal as in the proof of Theorem \ref{MT2-general-case}, we will work with the extended operator $\widetilde{\Sigma}_M$ defined in \eqref{eq:ext_invmap_sup1}. Theorem \ref{theorem:infiniteapprox} and Lemma \ref{lemma:L2_inv_sup1} implies that for all $\varepsilon > 0$, there exists $(d,\theta,\kappa) \in \mathcal N_{\sigma,7}^{V_{\mu} \to L^2(\Omega)}$ such that
			\[
			\int_{V_{\mu}} \norm{\widetilde{\Sigma}_M(T) - F^{d,\theta,\kappa}(T)}^2_{L^2(\Omega)} \nu(dT) < \varepsilon.
			\]
			Now, notice that
			\[
			\begin{aligned}
				\int_{\Lambda(Y_M(\Omega)) \cap V_{\mu}} \norm{\Sigma_M(T) - F^{d,\theta,\kappa}(T)}^2_{L^2(\Omega)} &= \int_{\Lambda(Y_M(\Omega)) \cap V_{\mu}} \norm{\widetilde{\Sigma}_M(T) - F^{d,\theta,\kappa}(T)}^2_{L^2(\Omega)}\\
				&\leq \int_{V_{\mu}} \norm{\widetilde{\Sigma}_M(T)-F^{d,\theta,\kappa}(T)}^2_{L^2(\Omega)} \nu(dT).
			\end{aligned}
			\]
			Therefore, for all $\varepsilon > 0$, there exists $(d,\theta,\kappa) \in \mathcal N_{\sigma,7}^{V_{\mu} \to L^2(\Omega)}$ such that
			\[
			\int_{\Lambda(Y_M(\Omega)) \cap V_{\mu}} \norm{\Sigma_M(T) - F^{d,\theta,\kappa}(T)}^2 \nu(dT) < \varepsilon.
			\]
			This ends the proof.
		\end{proof}
		
		\begin{remark}
			Although Theorem \ref{teo:DO_sup1} gives a DeepOnet that approximates the inverse Calderón's mapping with arbitrary accuracy, notice that the assumption \ref{sup:1} may be not easy to check.
		\end{remark}
		
		There is a second theorem that can be stated with the operator $\widetilde{\pi}_{M}$, providing an approximation of the inverse Calderón's mapping in a subset of the space of bilinear continuous operators from $H^{1/2}(\partial \Omega) \times H^{1/2}(\partial \Omega)$ into $\R$.\\
		
		\medskip
		In particular, the next Theorem implies that there exists a DeepOnet such that every conductivity $a \in L^{2}(\Omega)$ can be approximated by such DeepOnet with arbitrary accuracy if its associated Dirichlet-to-Neumann operator can define a bilinear continuous map from $H^{1/2}(\partial \Omega) \times H^{1/2}(\partial \Omega)$ into $\R$. 
	}
	
	{\color{black}
		\begin{theorem}
			Let $\mu$ and $\nu$ be probability measures on $H^{1/2}(\partial \Omega)$ and $W_{\mu}$, respectively,  with finite second moment. For $M \in (0,\infty)$ let $\Sigma_M$ be the inverse Calderón's mapping resctricted to $Y_M(\Omega)$, defined as in \eqref{eq:inverse_operator}. Then, for all $\varepsilon>0$ there exists $(d,\theta,\kappa)$ such that
			\begin{equation}
				\int_{D} \norm{\Sigma_M \circ \chi (B) - F^{d,\theta,\kappa}_\varepsilon (B)}_{L^2(\Omega)}^2 \nu(dB) < \varepsilon.
			\end{equation}
		\end{theorem}
		
		\begin{proof}
			Let $M \in (0,\infty)$. As equal as in the proof of Theorem \eqref{MT2-general-case} we will work with the extension of the desired operator. In particular, we work with the mapping $\widetilde{\pi}_M : W_{\mu} \to L^2(\Omega)$ defined in \eqref{eq:inverse_operator_BL}. Theorem \ref{theorem:infiniteapprox} and Lemma \ref{lemma:L2inv_BL} implies that for all $\varepsilon>0$, there exists $(d,\theta,\kappa) \in \mathcal N_{\sigma,7}^{W_{\mu} \to L^2(\Omega)}$ such that
			\[
			\int_{W_{\mu}} \norm{\widetilde{\pi}_M(B)-F^{d,\theta,\kappa}_{\varepsilon}(B)}^2_{L^2(\Omega)} \nu(dB) < \varepsilon.
			\]
			By the definition of $\widetilde{\pi}_M$ it follows that
			\[
			\begin{aligned}
				\int_{W_{\mu}} \norm{\widetilde{\pi}_M(B)-F^{d,\theta,\kappa}_{\varepsilon}(B)}^2_{L^2(\Omega)} \nu(dB) &= \int_{D \cap W_{\mu}} \norm{\Sigma_M \circ \chi (B) - F^{d,\theta,\kappa}_{\varepsilon}(B)}_{L^2(\Omega)}^2 \nu(dB) \\
				& \quad + \int_{D^c \cap W_{\mu}} \norm{F_{\varepsilon}^{d,\theta,\kappa}(B)}_{L^2(\Omega)}^{2} \nu(dB)\\
				& \geq  \int_{D \cap W_{\mu}} \norm{\Sigma_M \circ \chi (B) - F^{d,\theta,\kappa}_{\varepsilon}(B)}_{L^2(\Omega)}^2 \nu(dB).
			\end{aligned}
			\]
			Therefore, for all $\varepsilon>0$ there exists $(d,\theta,\kappa) \in \mathcal N_{\sigma,7}^{W_{\mu} \to L^2(\Omega)}$ such that
			\[
			\int_{D} \norm{\Sigma_M \circ \chi (B) - F^{d,\theta,\kappa}_{\varepsilon}(B)}_{L^2(\Omega)}^2 \nu(dB) < \varepsilon.
			\]
		\end{proof}
	}
}	

}

%

\section{Discussion}\label{Sect:6}

\subsection{Conclusions}

Theorems \ref{MT1}, \ref{MT2} {\color{black} and \ref{MT3}} provide suitable approximation results for the Dirichlet-to-Neumann operator and the {\color{black} direct and inverse} Calder\'on's mapping{\color{black}s} by DeepONets, an infinite dimensional version of the classical deep neural networks. These density results are universal, in the sense that the main tool used in this paper, Theorem \ref{theorem:infiniteapprox}, has a wide range of applications. In a soon-to-come, forthcoming publication, we will review the DNN approximation theory of other highly relevant dispersive operators, such as the \emph{KdV, NLS and Wave solution maps}. We conjecture that all these solution mappings are approached by well-chosen DeepONets.

\medskip

A very important aspect of our results is its general validity for the Calder\'on's mapping. We do not require specific approximations to the problem nor any suitable additional hypothesis other than further regularity of the conductivity $a$. The setting is general and has prospective applications to other Calder\'on's problems not considered in this paper. In this sense, Theorems \ref{MT1}, \ref{MT2} {\color{black} and \ref{MT3}} are general statements about the nature of the Calder\'on's problem. 

\medskip

A drawback of this method is the lack of a suitable way to explicitly compute the DeepONet with the current techniques. This is standard in a problem with such a level of generality, but it is also a desired improvement for forthcoming publications as well.

\subsection{Stability and reconstruction analysis}

From Alessandrini \cite{A}, see also \cite{Sal08}, it is well-known that the Calder\'on's problem enjoys very bad stability estimates, usually called logarithm estimates.  In our setting, from \eqref{eq:DO_Lambda_0} and \eqref{lomejor} we can infer less exact estimates on the variation of the Dirichlet-to-Neumann operator. However, on the other hand we reduce the stability estimate to a less complex problem, which is the understanding of the DeepONet found in \eqref{lomejor}. Additionally, the reconstruction of the conductivity $a$ is highly related to the reconstruction of the inverse of the found DeepONet. 

\subsection{Unsolved problems. Numerical aspects}

There are interesting aspects left open in this work that might be considered in forthcoming publications. First of all, there is the always interesting question of reproducing numerically the theoretical results obtained here. This is an interesting but not easy question but it requires a deep and long analysis that is out of the scope of this paper. Secondly, there is the question of expanding these results to other Calder\'on's operators, such as those considered in quasilinear problems (see e.g. \cite{MU20} and many references therein), or more complex geometries. Finally, related deep learning methods, such as the highly studied PINNs, could have theoretical implications in the efficient solvability of the Calder\'on's problem, for instance in the spirit of the results stated in \cite{PINNsZurich}. The relationship between PINNs and our results is still unknown to us. Once again, these are interesting open problems that may be considered elsewhere.

\bibliographystyle{alpha}

\newcommand{\etalchar}[1]{$^{#1}$}


\begin{thebibliography}{VAK{\etalchar{+}}22}

\bibitem[AVLR21]{AVLR21} G. Alberti, E. De Vito, M. Lassas, L. Ratti, M. Santacesaria: Learning the optimal Tikhonov regularizer for inverse problems. {\it Advances in Neural Information Processing Systems 34} (NeurIPS 2021).

\bibitem[Al90]{A}
    {   G. Alessandrini}. Singular solutions of elliptic equations and the determination of conductivity by boundary measurements, \textit{J. Diff. Equations} {\bf 84} (1990), 252--272.

\bibitem[ATYW19]{ATY+19}
Md.~Z. Alom, T. Taha, C. Yakopcic, S. Westberg, P.
  Sidike, M. Nasrin, M. Hasan, B. Essen, A. Awwal, and V. Asari. A state-of-the-art survey on deep learning theory and architectures.
  {\em Electronics}, 8:292, 03 2019.

\bibitem[AP06]{AP}
   {   K. Astala, and L. Pa\"ivar\"inta}. Calder\'on’s inverse conductivity problem in the plane, \textit{Annals of Math.}, {\bf 163} (2006), 265--299.

\bibitem[BKM21]{PINNsZurich} G. Bai, U. Koley, S. Mishra, and R. Molinaro, Physics Informed Neural Networks (PINNs) for approximating nonlinear dispersive PDEs.,  \emph{J. Comp. Math.}, 39 (2021), pp. 816-847.

\bibitem[BEJ19]{BEJ19}
C. Beck, W. E, and A. Jentzen. Machine learning approximation algorithms for high-dimensional fully
  nonlinear partial differential equations and second-order backward stochastic
  differential equations.
  {\em Journal of Nonlinear Science}, 29(4):1563--1619, jan 2019.

\bibitem[BHHJ20]{BHH+20}
C. Beck, F. Hornung, M. Hutzenthaler, A. Jentzen, and T. Kruse.
  Overcoming the curse of dimensionality in the numerical approximation
  of Allen{\textendash}Cahn partial differential equations via truncated
  full-history recursive multilevel picard approximations.
  {\em Journal of Numerical Mathematics}, 28(4):197--222, dec 2020.



\bibitem[BHK+21]{BHK+21} Kaushik Bhattacharya; Bamdad Hosseini; Nikola B. Kovachki; Andrew M. Stuart. Model Reduction And Neural Networks For Parametric PDEs. The SMAI Journal of computational mathematics, Volume 7 (2021), pp. 121-157. doi : 10.5802/smai-jcm.74. \url{https://smai-jcm.centre-mersenne.org/articles/10.5802/smai-jcm.74/}


\bibitem[Br96]{Br96}
    {   R. M. Brown}. Global uniqueness in the impedance-imaging problem for less regular conductivities, \textit{SIAM J. Math. Anal.} {\bf 27} (1996), 1049--1056.
    
    \bibitem[BU02]{BU}
    {   A. L. Bukhgeim, and G. Uhlmann}. Recovering a potential from partial Cauchy data, \textit{Comm. PDE}, {\bf 27} (3,4)(2002), 653--668.

\bibitem[Cal80]{C}
   {   A. P. Calder\'on}. On an inverse boundary value problem, \textit{Comp. Appl. Math.} {\bf 25} (1980), 133--138.
   
   \bibitem[Cal06]{Cal06}
A.~P. Calder\'on.
  On an inverse boundary value problem. {\em Comput. Appl. Math.}, 25(2--3):133--138, 2006.
  \textit{Seminar on inverse problems and applications} (Rio de
  Janeiro, 21--24 March 2006). Issue edited by G. Perla Menzala and G. Uhlmann.
  Reprinted from \textit{Seminar on numerical analysis and its applications to
  continuum physics} (1980). MR:2321646. Zbl:1182.35230.
   
   \bibitem[CaRo16]{CR}
   {   P. Caro, K. Rogers,} Global uniqueness for the Calder\'on Problem with Lipschitz conductitivies, \textit{Forum of Math., Pi} (2016), vol. 4.
   
   
\bibitem[Cas21]{Cas21}
J. Castro. Deep Learning schemes for parabolic nonlocal integro-differential equations, {\it Partial Differ. Equ. Appl.} 3 (2022), no. 6, 77.

\bibitem[Cas22]{Cas22}
J. Castro. The Kolmogorov infinite dimensional equation in a Hilbert space via deep learning methods. {\it J. Math. Anal. Appl.} 527 (2023), no. 2, Paper No. 127413.

   
  \bibitem[CC95]{CC95} T. Chen and H. Chen. Universal approximation to nonlinear operators by neural networks with arbitrary activation functions and its applications to dynamic systems. {\it IEEE Transactions on Neural Networks}, pages 911 - 917, 08 1995.
  
\bibitem[dHLW21]{dHLW21}  M. de Hoop, M. Lassas, C. Wong: Deep learning architectures for nonlinear operator functions and nonlinear inverse problems. \emph{Mathematical Statistics and Learning} 4 (2021), no. 1-2, 1--86.

\bibitem[EHJ17]{EHJ17}
W. E, J. Han, and A. Jentzen. Deep learning-based numerical methods for high-dimensional parabolic partial differential equations and backward stochastic differential
  equations. {\em Communications in Mathematics and Statistics}, 5(4):349--380,
  nov 2017. 
  
  \bibitem[EHJK21]{EHJ+21}
W. E, M. Hutzenthaler, A. Jentzen, and T. Kruse.
  Multilevel Picard iterations for solving smooth semilinear parabolic
  heat equations.
  {\em Partial Differential Equations and Applications}, 2(6), nov
  2021.
 
\bibitem[FY20]{FY20}
Y. Fan and L. Ying.
  Solving electrical impedance tomography with deep learning.
  {\em Journal of Computational Physics}, 404:109119, mar 2020.

\bibitem[GCC22]{GCC22}
R. Guo, S. Cao, and L. Chen;
  Transformer meets boundary value inverse problems, preprint arXiv (2022) \url{https://doi.org/10.48550/arXiv.2209.14977}.

\bibitem[HJE18]{HJE18}
J. Han, A. Jentzen, and W. E.
  Solving high-dimensional partial differential equations using deep
  learning.
  {\em Proceedings of the National Academy of Sciences},
  115(34):8505--8510, 2018.

 
   \bibitem[KSU07]{KSU} C.E. Kenig, J. Sj\"ostrand, G. Uhlmann, The Calder\'on problem with partial data, \textit{Annals of Math} 165 (2007), 567--591.

\bibitem[KoVo84]{KV}
R. Kohn, M. Vogelius, Determining conductivity by boundary measurements, \textit{Comm. Pure. Appl. Math.} {\bf 37} (1984), 289--298.
    
\bibitem[Hab15]{Hab15} B. Haberman, \emph{Uniqueness in Calder\'on's problem for conductivities with unbounded gradient}, Comm. Math. Phys. Volume 340, Issue 2 (2015), Page 639--659.    
    
\bibitem[HT13]{HT}
{   B. Haberman, D. Tataru}, Uniqueness in Calder\'on’s problem with Lipschitz conductivities.
\textit{Duke Math. J.}, {\bf 162} (2013), 497–516.

\bibitem[Hor91]{Hor91}
K. Hornik.
  Approximation capabilities of multilayer feedforward networks.  {\em Neural Networks}, 4(2):251--257, 1991.

\bibitem[HPW19]{HPW19}
C. Hur\'e, H. Pham, and X. Warin. Deep backward schemes for high-dimensional nonlinear PDEs, {\it Math. Comp.} 89 (2020), no. 324, 1547--1579.

\bibitem[HJK20]{HJK+20} M Hutzenthaler, A. Jentzen, T. Kruse, T. A. Nguyen, and P. Von Wurstemberger. Overcoming the curse of dimensionality in the numerical approximation of semilinear parabolic partial differential equations. {\it Proceedings of the Royal Society A: Mathematical, Physical and Engineering Sciences}, 476(2244):20190630, dec 2020.

\bibitem[HJKN20]{HJKN20} M. Hutzenthaler, A. Jentzen, T. Kruse, and T. A. Nguyen, A proof that rectified deep neural networks overcome the curse of dimensionality in the numerical approximation of semilinear heat equations, \emph{Partial Differ. Equ. Appl.} 1, 10 (2020).

\bibitem[Is93]{Isakov} V. Isakov, \emph{On uniqueness in inverse problems for semilinear parabolic equations}, {\it Arch. Rat. Mech. Anal.}, 124(1993) , 1--12.

\bibitem[LLP00]{LLP} I. E. Lagaris, A. Likas, and D. G. Papageorgiou. Neural-network methods for boundary value problems with
irregular boundaries. {\it IEEE Transactions on Neural Networks,} 11:1041-1049, 2000.

\bibitem[LLF98]{LLF} I. E. Lagaris, A. Likas, and D. I. Fotiadis. Artificial neural networks for solving ordinary and partial differential equations. {\it IEEE Transactions on Neural Networks}, 9(5):987-1000, 1998.


\bibitem[LMK22]{LMK21}
{\color{black} S. Lanthaler, S. Mishra, and G.~E. Karniadakis.
  Error estimates for deeponets: A deep learning framework in infinite
  dimensions.
  {\em Transactions of Mathematics and Its Applications}, vol. 6, no. 1, Jan. 2022.
}

\bibitem[LLPS93]{LLPS93} M. Leshno, I. V. Y. Lin, A. Pinkus, and S. Schocken. Multilayer Feedforward Networks With a Nonpolynomial
Activation Function Can Approximate Any Function, \emph{Neural Networks}, Vol. 6, pp. 861--867 (1993).

\bibitem[Li20]{Li20} Housen Li et al 2020, NETT: solving inverse problems with deep neural
networks. {\it Inverse Problems} 36 065005.


\bibitem[LJK19]{LJK19}
L. Lu, P. Jin, and G.~E. Karniadakis.
  DeepONet: Learning nonlinear operators for identifying differential
  equations based on the universal approximation theorem of operators.
  {\em CoRR}, abs/1910.03193, 2019.

%
%







%


















	

	


	

		
	
 \bibitem[MNdH20]{MNdH20} M. Magill, A. M. Nagel, and H. W. de Haan. Neural network solutions to differential equations in nonconvex domains: Solving the electric field in the slit-well microfluidic device. {\it Physical Review Research}, 2(3), jul 2020.

\bibitem[MJK20]{MJK20} Z. Mao, A. D. Jagtap, and G. Em Karniadakis. Physics-informed neural networks for high-speed flows. {\it Computer Methods in Applied Mechanics and Engineering}, 360:112789, 2020.

{\color{black}
\bibitem[MYEM23]{Moli} R. Molinaro, Y. Yang, B. Engquist, and S. Mishra,  {\it Neural Inverse Operators for Solving PDE Inverse Problems}, Proceedings of the 40th International Conference on Machine Learning (ICML, 2023).
}

\bibitem[MU20]{MU20} C. Mu\~noz, and G. Uhlmann. The Calder\'on problem for quasilinear elliptic equations, \emph{Annales de l'{}Institut Henri Poincar\'e C, Analyse non lin\'eaire} 
Volume 37, Issue 5, September--October 2020, Pages 1143-1166.


\bibitem[Na88]{Nac88}
A.~I. Nachman. Reconstructions from boundary measurements. {\em Annals of Mathematics}, 128(3):531--576, 1988.
  
 \bibitem[OJM22]{OJM22} G. Ongie, A. Jalal, C. A. Metzler, R. G. Baraniuk, A. G. Dimakis, and R. Willettk. Deep Learning Techniques
for Inverse Problems in Imaging, preprint arXiv 2020 \url{https://arxiv.org/pdf/2005.06001.pdf}.

\bibitem[PLK19]{PLK19}
G. Pang, Lu~Lu, and G. E.  Karniadakis.
  FPINNS: Fractional physics-informed neural networks.
  {\em SIAM Journal on Scientific Computing}, 41(4):A2603--A2626, 2019.


\bibitem[Par05]{Par05}
K.R. Parthasarathy. Probability measures on metric spaces, vol. 352.  {\it American Mathematical Soc.}, 2005.


\bibitem[PKL22]{PKL22} M. Puthawala, K. Kothari, M. Lassas, I. Dokmanic, M. de Hoop: Globally Injective ReLU Networks. To appear in {\it Journal of Machine Learning Research (JMLR)} (2022). 

\bibitem[RK18]{RK18}
M. Raissi and G.~E. Karniadakis.
  Hidden physics models: Machine learning of nonlinear partial
  differential equations. {\em Journal of Computational Physics}, 357:125--141, 2018.
  
  \bibitem[RPK19]{RPK19}
M.~Raissi, P.~Perdikaris, and G.E. Karniadakis.
  Physics-informed neural networks: A deep learning framework for
  solving forward and inverse problems involving nonlinear partial differential
  equations.  {\em Journal of Computational Physics}, 378:686--707, 2019.


\bibitem[Ro58]{Ros58}
F.~Rosenblatt.
  The perceptron: A probabilistic model for information storage and
  organization in the brain.  {\em Psychological Review}, 65(6):386--408, 1958.


\bibitem[Sal08]{Sal08}
M. Salo. {\em Calderón problem lecture notes}.
  2008.

\bibitem[San91]{San91}
I.W. Sandberg.
  Approximation theorems for discrete-time systems.
  {\em IEEE Transactions on Circuits and Systems}, 38(5):564--566,
  1991.

\bibitem[SS18]{SS18} J. Sirignano and K. Spiliopoulos. DGM: A deep learning algorithm for solving partial differential equations. {\it Journal of Computational Physics}, 375:1339-1364, dec 2018.













	









\bibitem[SU86]{SU86}
J. Sylvester and G. Uhlmann.
  A uniqueness theorem for an inverse boundary value problem in
  electrical prospection.
  {\em Comm. Pure Applied Math.}, 39(1): 91--112,
  1986.

\bibitem[SU87]{SU87}
J. Sylvester and G. Uhlmann.
  A global uniqueness theorem for an inverse boundary value problem.
  {\em Annals of Mathematics}, 125(1):153--169, 1987.






\bibitem[Gu09]{Gu_survey} G. Uhlmann. Electrical impedance tomography and Calder\'on's problem,  \emph{Inverse Problems}, Vol. 25, no. 12 (2009) 123011 (\url{http://dx.doi.org/10.1088/0266-5611/25/12/123011})

    \bibitem[UY10]{UY}    G. Uhlmann, M. Yamamoto, The Calder\'on problem with partial data in two dimensions,
    \textit{J. Amer. Math. Soc.} {\bf 23} (2010), 655--691.


\bibitem[VAK{\etalchar{+}}22]{VAK+22}
A. Vadeboncoeur, Ö.D. Akyildiz, I. Kazlauskaite, M. Girolami,
  and F. Cirak.  Deep probabilistic models for forward and inverse problems in
  parametric PDEs, \url{https://arxiv.org/abs/2208.04856} (2022).



	
\bibitem[Val22]{Val22} N. Valenzuela. A new approach for the fractional Laplacian via deep neural networks. Preprint arXiv:2205.05229 \url{https://arxiv.org/abs/2205.05229} (2022).	

{\color{black}
\bibitem[Val23]{Val23} N. Valenzuela. A numerical approach for the fractional Laplacian via deep neural networks. Preprint arXiv:2308.16272 \url{https://arxiv.org/abs/2308.16272} (2023).	
}
	
	
	



\bibitem[Ya16]{Yar16} D. Yarotsky. Error bounds for approximations with deep ReLu networks, {\it Neural Networks} Volume 94, October 2017, Pages 103--114.
	
	
	


\end{thebibliography}

\end{document}